\documentclass{article}

\usepackage[T1]{fontenc}
\usepackage[utf8]{inputenc}
\usepackage[english]{babel}
\usepackage[intlimits]{amsmath}
\usepackage{amsthm}
\usepackage{amssymb}
\usepackage{mathtools} 

\usepackage{tikz-cd}

\usepackage{lmodern}
\usepackage{bm} 
\usepackage{enumerate}

\usepackage[margin=1in]{geometry} 
\usepackage[hidelinks,pdfpagemode=UseNone,pdfstartview=FitH]{hyperref}
\usepackage{url}
\usepackage[autostyle]{csquotes}
\usepackage[safeinputenc, doi=false]{biblatex}
\theoremstyle{plain}
\newtheorem{theorem}{Theorem}[section]
\newtheorem{proposition}[theorem]{Proposition}
\newtheorem*{theorem*}{Theorem}
\newtheorem{corollary}[theorem]{Corollary}

\newenvironment{introcorollary}[1]
  {\par\addvspace{\topsep}\noindent
   \textbf{Corollary~\ref{#1}.}\itshape}
  {\par\addvspace{\topsep}}

\newenvironment{introtheorem}[1]
  {\par\addvspace{\topsep}\noindent
   \textbf{Theorem~\ref{#1}.}\itshape}
  {\par\addvspace{\topsep}}
  
\newtheorem{lemma}[theorem]{Lemma}

\theoremstyle{definition}

\newtheorem{remark}[theorem]{Remark}

\allowdisplaybreaks[1]

\newcommand*{\Rset}{\mathbb{R}}  
  
\newcommand*{\abs}[1]{\left\lvert#1\right\rvert}   % itseisarvo
\newcommand\twoheaddownarrow{\mathrel{\rotatebox[origin=c]{270}{$\twoheadrightarrow$}}}

\newcommand{\Vect}{\mathbf{Vect}}

\newcommand{\Mod}{\mathbf{Mod}}
\newcommand{\posetC}{\mathbf{P}}

\newcommand{\Sh}{\mathbf{Sh}}
 
\newcommand*{\Eph}{\mathbf{Eph}}

\newcommand{\Top}{\mathbf{Top}}
\newcommand{\SobC}{\mathbf{Sob}}

\DeclareMathOperator{\id}{id}
\DeclareMathOperator{\Hom}{Hom}

\DeclareMathOperator{\Fun}{Fun}
\DeclareMathOperator{\Open}{Open}
\DeclareMathOperator{\res}{res}

\DeclareMathOperator{\ImF}{Im}

\DeclareMathOperator{\Supp}{Supp}

\DeclareMathOperator{\supp}{supp}

\DeclareMathOperator{\Int}{Int}

\DeclareMathOperator{\Sob}{Sob}
\DeclareMathOperator{\InjSpec}{InjSpec}
\DeclareMathOperator{\Idl}{Idl}
\DeclareMathOperator{\Sk}{Sk}
\DeclareMathOperator{\Sp}{Sp}
\DeclareMathOperator{\Zg}{Zg}
\DeclareMathOperator{\ASupp}{ASupp}
\DeclareMathOperator{\ASpec}{ASpec}
\DeclareMathOperator{\sky}{sky}
\DeclareMathOperator{\fp}{fp}

\begin{document}

\title{Gabriel Spectrum of Persistence Categories}

\author{
Manu Harsu \\ manuh.research@gmail.com \and 
Eero Hyry \\ eero.hyry@tuni.fi}

\maketitle

%%%%%%%%%%%%%%%%%%%%%%%%%%%%%%%%%%%%%%%%%%%%%%%%%%%%%%%%

\begin{abstract}

\noindent

We determine the Gabriel spectrum of a category of sheaves of vector spaces in purely topological terms. For every topological space $X$, we prove that the Gabriel spectrum of $\Sh(X)$ is homeomorphic to $\Sk(\Sob(X))$, where $\Sob(X)$ denotes the sobrification of $X$ and $\Sk$ indicates passage to the Skula topology. The proof is based on a classification of indecomposable injective sheaves and a characterization of localizing subcategories in terms of Skula-open subsets. We show that the Gabriel spectrum is always Hausdorff, zero-dimensional, and totally disconnected, and that it is compact if and only if $X$ is Noetherian. This leads to computations of Gabriel spectra arising in persistence theory. In particular, the Gabriel spectrum of the category of persistence modules over $\Rset^n$ is homeomorphic to the space of ideals of $\Rset^n$ with a natural topology.

\end{abstract}

%\tableofcontents

%%%%%%%%%%%%%%%%%%%%%%%%%%%%%%%%%%%%%%%%%%%%%%%%%%%%%%%%%
\section{Introduction}

One of the central achievements of topological data analysis is the barcode representation of 
one-parameter persistence modules. The barcode theorem shows that pointwise finite-dimensional persistence modules indexed by the ordered real line decompose into interval modules, allowing algebraic information to be encoded by collections of intervals. This interval geometry lies at the heart of the success of persistent homology. 

In the multiparameter setting, however, interval decompositions no longer exist in general. Consequently, much recent work has focused on alternative structures for describing persistence modules, many of which are formulated in terms of rectangular regions of parameter space (see \cite{Botnan2023Introduction}, for example). The present article investigates a complementary direction. Rather than studying individual persistence modules, we study geometric structures associated with the ambient categories in which they reside.

A natural way to attach geometry to a Grothendieck category is through its Gabriel spectrum. Much as the prime spectrum encodes the geometry of a commutative ring, the Gabriel spectrum encodes the geometry of a Grothendieck category. If $\mathcal G$ is a Grothendieck category, then its Gabriel spectrum $\Sp(\mathcal G)$ is a topological space whose points are represented by indecomposable injective objects and whose topology reflects the localization theory of the category.

Let $P$ be a poset, and let $\posetC$ denote the thin category associated with $P$. Let $\Vect$ be the category of vector spaces over a field $k$. A persistence module over $P$ is a covariant functor  $\posetC\to\Vect$. Thus the category of persistence modules is the functor category $\Fun(\posetC, \Vect)$. It is well known that \[ \Fun(\posetC, \Vect)\simeq \Sh(P^a), \] where $P^a$ denotes $P$ endowed with its Alexandrov topology. 

Assume now that $P$ is a continuous poset in the sense of the domain theory. In \cite[Theorem 3.18]{EphModAndScottShOverContPoset}, we showed the categories of upper and lower semi-continuous persistence modules are both equivalent to the category of sheaves $\Sh(P^\sigma)$, where $P^\sigma$ denotes $P$ equipped with the Scott topology. Moreover, Theorem 4.17 of \cite{EphModAndScottShOverContPoset} identifies this category as the localization \[ \Sh(P^\sigma) \simeq \Sh(P^a)/\Eph, \] where $\Eph$ denotes the subcategory of ephemeral persistence modules. 

The realization of persistence categories as sheaf categories suggests investigating a more general question: 
\begin{quote}
For a topological space $X$, can the Gabriel spectrum of the category $\Sh(X)$ of sheaves of vector spaces on $X$ be described directly in terms of the topology of $X$? 
\end{quote}
The main result of this article provides a complete answer.

\begin{introcorollary}{Cor:GabrielIsSkulaSob}\emph{ (Main Theorem)}
For every topological space $X$, there is a natural homeomorphism \[ \Sp(\Sh(X))  \approx  \Sk(\Sob(X)), \] where $\Sob(X)$ denotes the sobrification of $X$ and 
$\Sk$ indicates passage to the Skula topology.
\end{introcorollary}
This result provides a purely topological description of the Gabriel spectrum. It involves two classical topological constructions that may be less familiar to readers from topological data analysis. 
The first is the sobrification $\Sob(X)$ of a topological space $X$. The points of $\Sob(X)$ are naturally identified with the irreducible closed subsets of $X$. From the viewpoint of sheaf theory, this is a natural construction because sheaves cannot distinguish a space from its sobrification. One may therefore think of $\Sob(X)$ as a canonical sober replacement of $X$ that preserves the sheaf theory of the space. The second construction is the Skula topology. This is obtained by refining the topology of a space so that every closed subset becomes open as well. Equivalently, it is the smallest topology containing both the open sets of the original space and their complements.

In order to proof the Main Theorem we first identify the points of the spectrum in Theorem \ref{Thm:IndInjAreSkysrc}. Let $X$ be a sober topological space. A sheaf $\mathcal{F}\in\Sh(X)$  is indecomposable injective if and only if $\mathcal{F}\cong\sky(x)$ for some point $x\in X$. Here $\sky(x)$ denotes the skyscraper sheaf at $x$ with value $k$.  
Since the points of $\Sob(X)$ are precisely the irreducible closed subsets of $X$, it follows that for an arbitrary topological space $X$ there is a natural bijection between indecomposable injective sheaves on $X$ and points of $\Sob(X)$. This result may be viewed as a sheaf-theoretic analogue of Höppner's classification of indecomposable injective persistence modules (see Remark \ref{Rmk:Hoppner}). 

The second ingredient of the proof is a description of the topology. The preceding theorem yields a natural bijection $\Sob(X)\to \Sp(\Sh(X))$. We then show that under this bijection the Gabriel topology coincides with the Skula topology. For this, we prove in Proposition~\ref{Prop:SkulaOpenCorrespondsSupports} that the Skula topology admits a direct categorical interpretation in terms of localization theory. Specifically, if $X$ is a topological space and $Z\subseteq X$, then $Z$ is Skula-open if and only if there exists a localizing subcategory $\mathcal L\subseteq \Sh(X)$ such that \[ Z=\Supp(\mathcal L) \coloneq \bigcup_{\mathcal F\in\mathcal L}\Supp(\mathcal F). \] 

The spectral spaces arising in this way enjoy strong topological properties. We prove in Theorem \ref{Thm:GabrielSpectrumTopProperties} that, for every topological space $X$, the Gabriel spectrum $\Sp(\Sh(X))$ is Hausdorff and admits a basis of clopen subsets. Consequently, it is totally disconnected and zero-dimensional in the sense of small inductive dimension. 
For topological spaces $X$ and $Y$ there is a natural homeomorphism \[ \Sp(\Sh(X\times Y)) \approx \Sp(\Sh(X))\times\Sp(\Sh(Y))\]
(Theorem \ref{Thm:ProductOfGabrielSpectrums}).
We also show that $\Sp(\Sh(X))$ is compact if and only if $X$ is Noetherian (Theorem \ref{Thm:GabrielSpecCompact}). Finally, if $X$ is Hausdorff, then $\Sp(\Sh(X))$ is discrete and has cardinality $\abs{X}$ (Theorem \ref{Thm:GabrielForHausdorff}).

We now return to persistence theory. Applying the Main Theorem to Scott spaces of continuous posets yields explicit descriptions of spectra of persistence categories. 

We first consider the Scott spectrum.

\begin{introtheorem}{Thm:GabrielSpecOfScottShOverR} 
Let $P=\Rset^n$ equipped with its coordinatewise order. Then 
there is an homeomorphism
\[ \Sp(\Sh((\Rset^n)^\sigma)) \approx (\Rset\cup\{\infty\})^n, \] where the right-hand side is endowed with the product Sorgenfrey topology. 
\end{introtheorem} 
The basic open subsets are the half-open rectangles $\left] x_1, y_1 \right] \times \dots \times \left] x_n, y_n \right]$. Thus the spectral topology is generated by the same class of rectangles that appears throughout multiparameter persistence theory. 

We also obtain an explicit description of the spectrum of ordinary persistence modules:

\begin{introtheorem}{Thm:GabrielSpecOfPerModOverR} 
Let $P=\Rset^n$. Then there is a homeomorphism \[ \Sp(\Sh(\Fun(\posetC, \Vect))) \approx \Idl(\Rset^n), \] where  $\Idl(\Rset^n)$ denotes the set of ideals of $\Rset^n$, endowed with the topology generated by the rectangles \[ \left]F_1,F_1'\right]\times\cdots\times \left]F_n,F_n'\right], \] where $F_i,F_i'\in\Idl(\Rset)$ for $i=1,\dots,n$. 
\end{introtheorem} 

Since every ideal of $\Rset^n$ is uniquely of the form $F_1\times\cdots\times F_n$, where each $F_i$ is an ideal of $\Rset$, the spectrum admits a particularly explicit description. Indeed, the ideals of $\Rset$ are precisely the intervals $]-\infty,x[$ and $]-\infty,x]$ for $x\in \Rset$ and $\Rset$ itself.

The relationship between ordinary persistence modules, Scott sheaves, and ephemeral modules is reflected at the level of spectra. We show in Remark~\ref{Rmk:Embedding} that, for a continuous poset $P$, the spectrum $\Sp(\Eph)$ is an open subspace of $\Sp(\Fun(\posetC,\Vect))$, while the complementary closed subspace is naturally homeomorphic to $\Sp(\Sh(P^\sigma))$. Thus the Scott spectrum is obtained from the spectrum of ordinary persistence modules by removing the open subset corresponding to ephemeral information.

Finally, we investigate the relationship between spectra and localization theory. For $P=\Rset\cup\{\infty\}$, we establish in Corollary \ref{Cor:LocalizingSubCatAndSkulaOpenR} natural bijections between localizing subcategories of $\Sh(P^\sigma)$, Skula-open subsets of $P^\sigma$, and Gabriel-open subsets of $\Sp(\Sh(P^\sigma))$. 
Consequently, the lattice of localizing subcategories of $\Sh(P^\sigma)$ can be recovered entirely from the topology of its Gabriel spectrum.

There are several notions of spectra associated with Grothendieck categories. For an arbitrary topological space $X$, we show in Theorem~\ref{Thm:AtomSpecHomeoGabriel} that the Gabriel spectrum of $\Sh(X)$ is naturally homeomorphic to the atom spectrum. Since it is easy to see that $\Sh(\Rset^\sigma)$ is neither locally coherent nor even locally finitely presented (see Proposition \ref{Prop:Easy}), we work throughout with the Gabriel spectrum, which does not require such assumptions.

For ordinary persistence modules, the situation is different. Lerch showed in \cite{LerchSpectrum} that the category of persistence modules $\Fun(\Rset,\Vect)$ is locally coherent. Endowing its injective spectrum with the Ziegler topology, he proved that the resulting space is homeomorphic to the space of ideals $\Idl(\Rset)$ equipped with the order topology. In particular, the Ziegler topology is strictly coarser than the Gabriel topology. For simplicity, we refer to this space as the Ziegler spectrum.

Since $\Sh(\Rset^\sigma)$ is not locally coherent (see Proposition \ref{Prop:Easy}), there is no genuine Ziegler spectrum associated with Scott sheaves on $\Rset^\sigma$. However, by \cite[Theorem~3.18]{EphModAndScottShOverContPoset}, the category $\Sh(\Rset^\sigma)$ is equivalent to the category of lower semi-continuous persistence modules $\Fun_c(\Rset,\Vect)$. We show in Corollary~\ref{Cor:OpenIndZieglerHomeoGabriel} that the subspace \[ \Zg(\Fun(\Rset,\Vect)) \cap \Fun_c(\Rset,\Vect) \] of the Ziegler spectrum is homeomorphic to the Gabriel spectrum $\Sp(\Sh(\Rset^\sigma))$.  Thus the Gabriel spectrum of Scott sheaves can be recovered as a natural subspace of the Ziegler spectrum of ordinary persistence modules.

%%%%%%%%%%%%%%%%%%%%%%%%%%%%%%%%%%%%%%%%%%%%%%%%%%%%%%%%%%%%%%%55
\section{Preliminaries on Spectra}

%%%%%%%%%%%%%%%%%%%%%%%%%%%%%%%%%%%%%%%%%%%%%%%%%%%%%%%%%%%%
\subsection{Injective Spectrum}

Throughout this section, let $\mathcal{G}$ be a Grothendieck category. 
Recall that a full subcategory $\mathcal{C}$ of $\mathcal{G}$ is called a \emph{Serre subcategory}, if for every short exact sequence $0 \to M' \to M \to M'' \to 0$, one has $M \in \mathcal{C}$ if and only if $M', M'' \in \mathcal{C}$. 
For every Serre subcategory $\mathcal{C}$, one can form the quotient category $\mathcal{G} / \mathcal{C}$. 
A Serre subcategory is \emph{localizing}, if the quotient functor $\mathcal{G} \to \mathcal{G} / \mathcal{C}$ admits a right adjoint, often called a \emph{section}. 
Equivalently, $\mathcal{C}$ is closed under arbitrary direct sums. 
A localizing subcategory is the torsion class of a hereditary torsion theory. 
In particular, there is an associated a torsion functor 
$t_{\mathcal{L}} \colon \mathcal{G} \to \mathcal{L}$. 
For a general reference, see \cite{popescu}. 

The collection of isomorphism classes of the indecomposable injective objects of $\mathcal{G}$ forms a set, since every indecomposable injective is isomorphic
to the injective envelope of a quotient of a generator of $\mathcal{G}$. 
This set is called the \emph{injective spectrum} of $\mathcal{G}$ and is denoted by $\InjSpec(\mathcal{G})$.

%%%%%%%%%%%%%%%%%%%%%%%%%%%%%%%%%%%%%%%%%%%%%%%%%%%%%%%%%%
\subsection{Gabriel Spectrum}

The first topology on the injective spectrum that we consider is the Gabriel topology.  
This topology has appeared, for example, in \cite{HerzogGabrielSpectrum} and \cite{hrbek2025mutationgabrielspectrum}. 
Before introducing the Gabriel spectrum, we need to recall that the localizing subcategories of $\mathcal{G}$ form a complete lattice. 
The collection of localizing subcategories of $\mathcal{G}$ is a set rather than a proper class. 
Indeed, a localizing subcategory is determined by the quotients of a fixed generator of $\mathcal{G}$ that it contains.
It is obvious that an arbitrary intersection of localizing subcategories is again a localizing subcategory. 
If $(\mathcal{L}_{\alpha})_{\alpha \in A}$ is a family of localizing subcategory of $\mathcal{G}$, then the torsion functor associated with their intersection is given by \[t_{\cap \mathcal{L}_{\alpha}}(M) = \bigcap_{\alpha \in A} t_{\mathcal{L}_{\alpha}}(M)\] for every $M \in \mathcal{G}$. 
Thus, the set of localizing subcategories of $\mathcal{G}$ form a poset that is ordered by inclusion, and all meets exist. 
Consequently, this poset is also a complete lattice. 
The join of localizing subcategories $\mathcal{L}_{\alpha}$ ($\alpha \in A$) is the smallest localizing subcategory containing every $\mathcal{L}_{\alpha}$, and is denoted by $\bigvee_{\alpha \in A} \mathcal{L}_{\alpha}$. 
By \cite[Lemma 2.2]{krause2024booleanspectrumgrothendieckcategory}, the corresponding torsion functor satisfies 
\[
  t_{\lor \mathcal{L}_{\alpha}}(M) = \sum_{\alpha \in A} t_{\mathcal{L}_{\alpha}}(M) 
\]
for every object $M \in \mathcal{G}$, and $t_{\lor \mathcal{L}_{\alpha}}(M) = 0$ if and only if $t_{\mathcal{L}_{\alpha}}(M) = 0$ for all $\alpha \in A$.

For a localizing subcategory $\mathcal{L}$ of $\mathcal{G}$, we write 
\[
  \mathcal{O}(\mathcal{L}) 
  = \{ E \in \InjSpec(\mathcal{G}) \mid \Hom(M, E) \not= 0 \text{ for some } M \in \mathcal{L} \} 
  = \{ E \in \InjSpec(\mathcal{G}) \mid t_{\mathcal{L}}(E) \not= 0 \}. 
\]
Obviously, $\mathcal{O}(\mathcal{G}) = \InjSpec(\mathcal{G})$ and $\mathcal{O}(0) = \emptyset$. 
Moreover, it follows immediately from the descriptions of the torsion functors associated with joins and meets that 
\[
\mathcal{O} \left(\bigvee_{\alpha \in A} \mathcal{L}_{\alpha} \right) 
= \bigcup_{\alpha \in A} \mathcal{O}(\mathcal{L}_{\alpha}) 
\quad \text{and} \quad 
\mathcal{O} \left(\bigcap_{\alpha \in A} \mathcal{L}_{\alpha} \right) 
\subseteq \bigcap_{\alpha \in A} \mathcal{O}(\mathcal{L}_{\alpha})
\]
for every family of localizing subcategories ($\mathcal{L}_{\alpha})_{\alpha \in A}$. 
For the finite intersections, the converse inclusion is also holds: $\mathcal{O}(\mathcal{L}_1) \cap \mathcal{O}(\mathcal{L}_2) = \mathcal{O}(\mathcal{L}_1 \cap \mathcal{L}_2)$ for localizing subcategories $\mathcal{L}_1$ and $\mathcal{L}_2$ of $\mathcal{G}$. 
This is because, if $t_{\mathcal{L}_1}(E) \not= 0$ and $t_{\mathcal{L}_2}(E) \not= 0$ for $E \in \InjSpec(\mathcal{G})$, then both are essential subobjects of $E$, so their intersection $t_{\mathcal{L}_1}(E) \cap t_{\mathcal{L}_2}(E)$ must also be non-zero. 
For more details, see \cite[Lemma 4.8]{PAPPACENA2002559}. 

Consequently, the collection $\{ \mathcal{O}(\mathcal{L}) \mid \mathcal{L} \text{ is a localizing subcategory of } \mathcal{G} \}$ forms a topology on $\InjSpec(\mathcal{G})$. 
It is called the \emph{Gabriel topology} of $\InjSpec(\mathcal{G})$, and the sets $\mathcal{O}(\mathcal{L})$ (where $\mathcal{L}$ is a localizing subcategory of $\mathcal{G}$) are called the \emph{basic Gabriel open sets}. 
The injective spectrum equipped with this topology is called the \emph{Gabriel spectrum} of $\mathcal{G}$.
To distinguish between a topological space and its underlying set, we write $\Sp(\mathcal{G})$ for the Gabriel spectrum of $\mathcal{G}$.

%%%%%%%%%%%%%%%%%%%%%%%%%%%%%%%%%%%%%%%%%%%%%%%%%%%%%%%%%%%
\subsection{Ziegler Spectrum}

Next, we endow the injective spectrum with the so-called Ziegler topology. 
For simplicity, we refer to the resulting topological space as the \emph{Ziegler spectrum}, although in the literature the term “Ziegler spectrum” usually refers to the pure-injective spectrum endowed  with the Ziegler topology. 
Note that this construction is defined only for locally coherent Grothendieck categories. 
For completeness, we recall the relevant terminology; for further details, see \cite{Herzog1997, Krause1997}.
An object $M \in \mathcal{G}$ is called \emph{finitely presented}, if the functor $\Hom(M, -)$ commutes with direct limits. 
The full subcategory of finitely presented objects is denoted by $\fp(\mathcal{G})$. 
If the isomorphism classes of $\fp(\mathcal{G})$ form a set and every object of $\mathcal{G}$ is a direct limit of finitely presented objects, then $\mathcal{G}$ is called a \emph{locally finitely presented category}. 
A locally finitely presented category $\mathcal{G}$ is \emph{locally coherent} provided that $\fp(\mathcal{G})$ is also Abelian. 

For rest of this subsection, we suppose that $\mathcal{G}$ is a locally coherent Grothendieck category. 
We follow \cite[Chapter 12]{KrauseHomTheoryRep}. 
For a class $\mathcal{C} \subseteq \fp(\mathcal{G})$ and for a subset $U \subseteq \InjSpec(\mathcal{G})$, define 
\begin{align*}
    \mathcal{C}^{\perp} &= \{ E \in \InjSpec(\mathcal{G}) \mid \Hom(C, E) = 0 \text{ for all } C \in \mathcal{C} \}, \\
    {}^{\perp} U &= \{ M \in \fp(\mathcal{G}) \mid \Hom(M, E) = 0 \text{ for all } E \in U \}. 
\end{align*}
The assignment $U \mapsto \overline{U} = ({}^{\perp}U)^{\perp}$, where $U \subseteq \InjSpec(\mathcal{G})$, defines a closure operator in the sense of Kuratowski. 
In particular, the subsets $U \subseteq \InjSpec(\mathcal{G})$ satisfying $U = \overline{U}$ form the closed subsets of a topology on $\InjSpec(\mathcal{G})$ (\cite[Lemma 12.1.12]{KrauseHomTheoryRep}). 
This topology is called the \emph{Ziegler topology}.

%%%%%%%%%%%%%%%%%%%%%%%%%%%%%%%%%%%%%%%%%%%%%%%%%%%%%%%%%%%
\subsection{Atom Spectrum}

In this subsection, we study another topological space called the atom spectrum. 
We begin by recalling some necessary terminology; see \cite{KandaAtomSpec1} and \cite{KandaAtomSpec3} for further background. 
Although many of the results discussed below hold more generally for Abelian categories, we restrict our attention to Grothendieck categories in order to avoid set-theoretic issues.

Let $\mathcal{G}$ be a Grothendieck category. A non-zero object $H \in \mathcal{G}$ is called \emph{monoform} if for every non-zero subobject $N$ of $H$, the objects $H$ and $H/N$ have no common non-zero subobject. Note that every non-zero subobject of a monoform object is again monoform \cite[Proposition 2.2]{KandaAtomSpec1}, and every monoform object is \emph{uniform} 
(or \emph{coirreducible}) meaning that any two non-zero subobjects have a non-zero intersection 
\cite[Proposition 2.6]{KandaAtomSpec1}. 
Two monoform objects are said to be \emph{atom-equivalent} if they admit a common non-zero subobject. 
Atom-equivalence defines an equivalence relation on the class of monoform objects \cite[Proposition 2.8]{KandaAtomSpec1}. 
We denote the equivalence class of a monoform object $H$ by $\overline{H}$ and call such an equivalence class an \emph{atom} of $\mathcal{G}$. 
The \emph{atom spectrum} $\ASpec(\mathcal{G})$ is the collection of all atoms of $\mathcal{G}$.
Since $\mathcal{G}$ is a Grothendieck category, the atom spectrum $\ASpec(\mathcal{G})$ is a set rather than a proper class \cite[Proposition 2.7]{KandaAtomSpec3}. 
The \emph{atom support} of an object $M \in \mathcal{G}$ is defined by
\[
  \ASupp(M) = \{ \overline{H} \mid H \text{ is a monoform subquotient of } M \}. 
\]
The collection $\{ \ASupp(M) \mid M \in \mathcal{G} \}$ forms a topology on $\ASpec(\mathcal{G})$ \cite[Proposition 3.2]{KandaAtomSpec3}. 
The atom spectrum will always be endowed with this topology. 

Next, we explain the relationship between the atom spectrum and the Gabriel spectrum. 
We denote the injective envelope of an object $M \in \mathcal{G}$ by $E(M)$. 
Two monoform objects $H, H' \in \mathcal{G}$ are atom-equivalent if and only if $E(H) \cong E(H')$ \cite[Lemma 5.8]{KandaAtomSpec1}. 
Furthermore, an object $M\in\mathcal G$ is uniform if and only if its injective envelope $E(M)$ is indecomposable \cite[p.~296--297]{popescu}. Consequently, the assignment $\overline H \mapsto E(H)$ induces a well-defined map $\overline E\colon \ASpec(\mathcal G)\to \InjSpec(\mathcal G)$. This map is always injective. Atom supports also admit a description in terms of injective envelopes. For every object $M\in\mathcal G$, \[ \ASupp(M) = \{\overline H\in\ASpec(\mathcal G)\mid \Hom(M,E(H))\neq 0\}. \]
See \cite[Lemma~3.2]{RezaAtomSpectrum} for details.

%%%%%%%%%%%%%%%%%%%%%%%%%%%%%%%%%%%%%%%%%%%%
\section{Sheaves}

%%%%%%%%%%%%%%%%%%%%%%%%%%%%%%%%%%%%%%%%%%%%%%%%%%
\subsection{Basic Notation}

For a general reference on sheaves, see \cite{SheavesOnManifolds}.
Let $X$ be a topological space and let $k$ be a field. 
We use the notation $\Sh(X)$ for the category of sheaves on $X$ with values in the category of $k$-vector spaces $\Vect$. 
It is well-known to be a Grothendieck category. 
Let $\mathcal{F}$ be a sheaf on $X$. 
The \emph{support} of $\mathcal{F}$ is the set $\Supp(\mathcal{F}) = \{ x \in X \mid \mathcal{F}_x \not= 0 \}$. 
If $U \subseteq X$ is an open subset, then the \emph{support} of a section $s \in \mathcal{F}(U)$ is  $\supp(s) = \{ x \in U \mid s_x \not= 0 \}$. 
Recall that the support of a section is a closed in $U$. 
In particular, it is locally closed subset of $X$ (i.e.\ an intersection of an open subset and a closed subset). 

Let $Z \subseteq X$ be a locally closed subset of a topological space $X$. 
Then every sheaf $\mathcal{F}$ on $Z$ admits an \emph{extension by zero}, which is a sheaf on $X$. 
If $i \colon Z \to X$ denotes the inclusion, we write this sheaf as $i_!\mathcal{F}$.  
Its restriction back to $Z$ (i.e.\ the sheaf $i^*i_! \mathcal{F}$) is $\mathcal{F}$ itself. 
Conversely, if $\mathcal{F}$ is a sheaf on $X$, then we define $\mathcal{F}_Z = i_!i^*\mathcal{F}$. 
If $Z' \subseteq X$ is another locally closed subset, then $(\mathcal{F}_Z)_{Z'} = \mathcal{F}_{Z \cap Z'}$. 
Moreover, if $Z$ is open, $\mathcal{F}_Z$ is a subsheaf of $\mathcal{F}$. 
If $k_X$ is the constant sheaf of $X$ (with the value $k$), we will denote by $k_Z$ the sheaf $(k_X)_Z$, for short. 
The subsheaves of $k_X$ are in one-to-one correspondence with the open subsets of $X$. 
It follows that the subsheaves of $k_Z$ are in one-to-one correspondence with the open subsets of $Z$. 
For further details, see \cite[Proposition 2.3.6]{SheavesOnManifolds}. 

Next, we prove a simple lemma. 
Recall that a topological space $X$ is called \emph{reducible} if it can be written as an union of two proper closed subsets, otherwise it is \emph{irreducible} (or \emph{hyperconnected}). 
Equivalently, $X$ is irreducible if and only if any two non-empty open subsets have non-empty intersection.

\begin{lemma}\label{Lemma:ConstantSheafProperties}
Let $X$ be a topological space, and let $Z \subseteq X$ be a non-empty locally closed set. 
Then the following are equivalent: 
\begin{enumerate}[(i)]
    \item $Z$ is irreducible; 
    \item $k_Z$ is uniform; 
    \item $k_Z$ is monoform. 
\end{enumerate}
\end{lemma}

\begin{proof}
(i) $\Leftrightarrow$ (ii): 
Simply, 
\begin{align*}
    k_Z \text{ is uniform} 
    &\iff \text{ all non-zero subobjects of $k_Z$ have a non-zero intersection} \\
    &\iff \text{ all non-empty open subsets of $Z$ have a non-empty intersection (in $Z$)} \\
    &\iff Z \text{ is irreducible}. 
\end{align*}

(i) $\Rightarrow$ (iii): 
Assume that $Z$ is irreducible.
To show that $k_Z$ is monoform, let $U \subseteq X$ be an open subset such that $U \cap Z \not= \emptyset$ and suppose that $k_Z$ and $k_{Z} / k_{U \cap Z} = k_{Z \setminus U}$ admit a common subobject.
In other words, there exist open subsets $V,W \subseteq X$ such that $V \cap Z = W \cap (Z \setminus U)$. 
Suppose that $V \cap Z \neq \emptyset$. Since $Z$ is irreducible and both $U \cap Z$ and $V \cap Z$ are non-empty open subsets of $Z$, we obtain $U \cap V \cap Z \not= \emptyset$. 
However,
\[
  V \cap Z \cap U = (W \cap (Z \setminus U)) \cap U = \emptyset, 
\]
which is a contradiction. 
Therefore, $V \cap Z = \emptyset$, and hence, the only common subobject of $k_Z$ and $k_Z/k_{U \cap Z}$ is zero. 
Thus, $k_Z$ is monoform. 

(iii) $\Rightarrow$ (ii) Every monoform object is also uniform. 
\end{proof}

Let $X$ be a topological space, and let $\mathcal{F}$ be a sheaf on $X$. 
Let $S = \{ s_i \in \mathcal{F}(U_i) \mid i \in I \}$ be some set of sections of $\mathcal{F}$, where $U_i \subseteq X$ is open for all $i \in I$. 
The \emph{subsheaf generated by $S$} is the smallest subsheaf $\mathcal{F}'$ of $\mathcal{F}$ such that $s_i \in \mathcal{F}'(U_i)$ for all $i \in I$. 
This subsheaf can also be defined as the image 
\[
  \ImF \left(\bigoplus_{i \in I} k_{U_i} \to \mathcal{F} \right). 
\]
If $S$ consist of exactly one section $s \in \mathcal{F}(X)$, then we denote by $\langle s \rangle$ the subsheaf generated by $s$. 
For the support, we have $\Supp(\langle s \rangle) = \supp(s)$. 
It is also worth noting that $\langle s \rangle \cong k_{\supp(s)}$, since $k$ is a field.

\begin{lemma}\label{Lemma:SectionInduceMonomorf}
Let $X$ be a topological space and let $\mathcal{F}$ be a sheaf on $X$. 
\begin{enumerate}[(i)]
  \item If $U \subseteq X$ is open and $s \in \mathcal{F}(U)$ is a non-zero section, then there exist a non-zero monomorphism $k_{\supp(s)} \to \mathcal{F}$ whose image contains $s$; 

  \item If $x \in \Supp(\mathcal{F})$, then there exist a locally closed set $Z \subseteq X$ containing $x$ and a non-zero monomoprhism $k_Z \to \mathcal{F}$. 
\end{enumerate}
\end{lemma}

\begin{proof}
(i) The section $s$ generates a subsheaf, which is isomorphic with $k_{\supp(s)}$. 
Note that $k_{\supp(s)}$ is well-defined, because $\supp(s)$ is a locally closed set of $X$. 

(ii) If $x \in \Supp(\mathcal{F})$, then there exist an open subset $U \subseteq X$ containing $x$ and a non-zero section $s \in \mathcal{F}(U)$ such that $s_x \not= 0$. 
Now, apply (i) for the section $s$. 
\end{proof}

%%%%%%%%%%%%%%%%%%%%%%%%%%%%%%%%%%%%%%%%%%%%%%%%%%%%%5
\subsection{Sheaves on a Sober Space}

Let $X$ be a topological space. 
If a closed irreducible subset $Z \subseteq X$ is the closure of a unique point, then that point is called the \emph{generic point} of $Z$. 
A topological space is called \emph{sober}, if every closed irreducible subset has a unique 
generic point. 
Every Hausdorff space is sober, since its only irreducible subsets are singletons, and every sober space is $T_0$. 

The \emph{sobrification} $\Sob(X)$ of $X$ is the set of all irreducible closed subsets of $X$, endowed with the so-called \emph{lower Vietoris topology}, whose open subsets are the sets 
\[
  \diamond U = \{ F \subseteq X \mid F \text{ is closed, irreducible, and } U \cap F \not= \emptyset \} 
\]
for each open subset $U \subseteq X$. 
The sobrification $\Sob(X)$, together with the map 
\[
  \eta \colon X \to \Sob(X), \quad \eta(x) = \overline{\{ x \}}, 
\]
satisfies the following universal property: 
For every continuous map $f \colon X \to Y$ to a sober topological space $Y$, there exists a unique continuous map $f^s \colon \Sob(X) \to Y$ such that $f = f^s \circ \eta$, 
\begin{center}
\begin{tikzcd}
  X \arrow{rr}{f} \arrow{rd}[swap]{\eta} & & Y \\
   & \Sob(X) \arrow[ru, dashed, swap]{}{f^s}
\end{tikzcd}
\end{center}
This universal property determines the sobrification uniquely up to unique isomorphism and makes it functorial. Consequently, the sobrification functor is left adjoint to the inclusion of the full subcategory of sober topological spaces into the category of topological spaces. In particular, there is a natural bijection 
\[
  \Hom_{\Top}(X,Y) = \Hom_{\SobC}(\Sob(X), Y)
\]
for every topological space $X$ and every sober topological space $Y$. 
For a reference on sobrification, see \cite[Chapter 8]{goubault-larrecq_2013}.

\begin{remark}[Sobrification preserves products]\label{Rmk:SobPresProd}
Observe that, although the sobrification functor is a left adjoint, it preserves products. 
More precisely, if $(X_i)_{i\in I}$ is a family of topological spaces, then
\[
  \Sob \left( \prod_{i \in I} X_i \right) = \prod_{i \in I} \Sob(X_i). 
\]
For more details, see \cite[Theorem 1.4]{sobrificationHoffmannRemainder}. 
\end{remark}

\begin{remark}\label{Rmk:SheavesOnSobEqv}
Recall that the category of sheaves is determined by the frame of open subsets $\Open(X)$ rather than by the topological space $X$ itself. 
More precisely, if $X_1$ and $X_2$ are topological spaces and $f \colon \Open(X_1) \to \Open(X_2)$ is an order-isomorphism, then it induces an equivalence of categories $\Sh(X_1) \simeq \Sh(X_2)$. 
In particular, this applies to a topological space $X$ and its sobrification: 
\[
  \Sh(X) \simeq \Sh(\Sob(X)). 
\]
Consequently, we may always restrict our attention to sober topological spaces. 
\end{remark}

%%%%%%%%%%%%%%%%%%%%%%%%%%%%%%%%%%%%%%%%%%%%%%%%
\subsection{Indecomposable Injective Sheaves}\label{Subsec:InjScottSh}

Let $X$ be a topological space. 
Recall that a sheaf $\mathcal{F}$ on $X$ is \emph{flabby} if the restriction morphisms $\res_{V,U}$ are surjective for all open subsets $U \subseteq V$. 
If $k$ is a field, then a sheaf $\mathcal{F}$ on $X$ is injective if and only if it is flabby 
(see \cite[Exercise II.10]{SheavesOnManifolds}). 
We now give an important example of an injective sheaf. 
Let $x \in X$. 
The \emph{skyscraper sheaf at $x$ with value $k$} (denoted by $\sky(x)$) is defined by setting 
\[
  \sky(x)(U) = 
  \begin{cases}
      k, &\text{if } x \in U; \\
      0, &\text{otherwise}
  \end{cases}
\]
for all open subsets $U \subseteq X$. 
It is well known that the support of $\sky(x)$ is $\overline{ \{ x \}}$, and that $\Hom(\mathcal{F}, \sky(x)) = \Hom(\mathcal{F}_x, k)$ for every sheaf $\mathcal{F}$ on $X$. 
Since $k$ is injective as a $k$-vector space, it follows that $\sky(x)$ is an injective sheaf on $X$. 
Moreover, $\sky(x)$ is clearly indecomposable. 
More generally:

\begin{lemma}\label{Lemma:IndicatorSheafFlabbyForClosedIrred}
Let $X$ be a topological space, and let $Z \subseteq X$ be a closed and irreducible subset. 
Then $k_Z$ is flabby and indecomposable. 
\end{lemma}

\begin{proof}
Firstly, $k_Z$ is uniform by Lemma \ref{Lemma:ConstantSheafProperties}, and hence, indecomposable. 
Secondly, the constant sheaf $k_Z$ on $Z$ is flabby, since every non-empty open subset of $Z$ is connected. 
The extension by zero along the closed inclusion $i \colon Z \to X$ coincides with the direct image functor. 
In particular, it preserves flabbiness. 
Hence, $k_Z$ is also flabby as a sheaf on $X$. 
\end{proof}

\begin{proposition}\label{Prop:SheafHomSuppSky}
Let $X$ be a topological space, and let $\mathcal{F}$ be a sheaf on $X$. 
Then $x \in \Supp(\mathcal{F})$ if and only if $\Hom(\mathcal{F}, \sky(x)) \not= 0$. 
\end{proposition}

\begin{proof}
If $x \in \Supp(\mathcal{F})$, then $\mathcal{F}_x \not= 0$, so there exist a non-zero $k$-linear map $\mathcal{F}_x \to k$. 
This induces a non-zero morphism of sheaves $\mathcal{F} \to \sky(x)$. 
Suppose then that $\Hom(\mathcal{F}, \sky(x)) \not= 0$. 
Then there exists a non-zero morphism $\varphi \colon \mathcal{F} \to \sky(x)$ and an open subset $U \subseteq X$ such that $\varphi_U \not= 0$. 
We have a commutative diagram 
\[
\begin{tikzcd}
  \mathcal{F}(U) \arrow{r}{\varphi_U} \arrow{d} & \sky(x)(U) \arrow{d} \\
  \mathcal{F}_x \arrow{r}{\varphi_x} & \sky(x)_x
\end{tikzcd}
\]
From this we can see that the non-zero morphism $\mathcal{F}(U) \to \sky(x)_x$ factors through $\mathcal{F}_x$. 
Thus, $\mathcal{F}_x \not= 0$.
\end{proof}

Next, we fully characterize the indecomposable injectives sheaves.

\begin{proposition}\label{Prop:IndInjSuppIrr}
Let $X$ be a topological space, and let $\mathcal{F}$ be a sheaf on $X$. 
If $\mathcal{F}$ is uniform, then $\Supp(\mathcal{F})$ is irreducible. 
In particular, if $\mathcal{F}$ is indecomposable injective sheaf on $X$, 
then $\Supp(\mathcal{F})$ is irreducible. 
\end{proposition}

\begin{proof}
Let $U_1, U_2 \subseteq X$ be open subsets such that $\Supp(\mathcal{F}) \cap U_1$ and $\Supp(\mathcal{F}) \cap U_2$ are non-empty. 
In particular, there exist points $x_1 \in \Supp(\mathcal{F}) \cap U_1$ and $x_2 \in \Supp(\mathcal{F}) \cap U_2$. 
Consider the subsheaves $\mathcal{F}_{U_1}$ and $\mathcal{F}_{U_2}$ of $\mathcal{F}$. 
These are non-zero, since $(\mathcal{F}_{U_i})_{x_i} = \mathcal{F}_{x_i} \not= 0$ for $i = 1,2$. 
As non-zero subsheaves of the uniform sheaf $\mathcal{F}$, their intersection must also be non-zero. 
However, their intersection is precisely $\mathcal{F}_{U_1 \cap U_2}$. 
Thus, $\mathcal{F}_y = (\mathcal{F}_{U_1 \cap U_2})_y \not= 0$ for some $y \in U_1 \cap U_2$ proving the first claim. 

Finally, every indecomposable injective object is uniform; see, for example, \cite[p.~296]{popescu}. This proves the second claim.
\end{proof}

\begin{proposition}\label{Pro:IndInjShIsGeneratedByAGlobalSec}
Let $X$ be a topological space, and let $\mathcal{F}$ be an indecomposable injective sheaf on $X$. 
Then $\mathcal{F}$ is generated by a global section i.e.\ 
there exists a non-zero global section $s \in \mathcal{F}(X)$ such that $\mathcal{F} = \langle s \rangle$. 
\end{proposition}

\begin{proof}
Take any non-zero global section $s \in \mathcal{F}(X)$ (such a section exists since $\mathcal{F}$ is non-zero and flabby). This section generates a subsheaf $\langle s \rangle$ of $\mathcal{F}$. The support of this subsheaf is $\supp(s)$ and, in particular, is closed.
Observe that $\langle s \rangle$ is uniform, as a subobject of the uniform sheaf $\mathcal{F}$. Hence, by Lemma \ref{Lemma:ConstantSheafProperties}, the support $\supp(s)$ is irreducible. It follows from Lemma \ref{Lemma:IndicatorSheafFlabbyForClosedIrred} that $\langle s \rangle$ is flabby, and therefore injective.
Consequently, the exact sequence $0 \to \langle s \rangle \to \mathcal{F} \to \mathcal{F} / \langle s \rangle \to 0$ splits. 
Since $\mathcal{F}$ is indecomposable, we must have either $\langle s \rangle = 0$ or $\langle s \rangle = \mathcal{F}$. As $s \neq 0$, the first possibility is excluded. Therefore, $ \mathcal{F} = \langle s \rangle$.
\end{proof}

\begin{proposition}\label{Prop:IndInjSuppClosed}
Let $X$ be a topological space, and let $\mathcal{F}$ be an indecomposable injective sheaf on $X$. 
Then the support of $\mathcal{F}$ is closed. 
\end{proposition}

\begin{proof}
By Proposition \ref{Pro:IndInjShIsGeneratedByAGlobalSec}, $\Supp(\mathcal{F}) = \supp(s)$ for some non-zero global section $s \in \mathcal{F}(X)$. 
Thus, $\Supp(\mathcal{F})$ is closed. 
\end{proof}

We are now ready to prove

\begin{theorem}\label{Thm:IndInjAreSkysrc}
Let $X$ be a sober topological space, and let $\mathcal{F}$ be a sheaf on $X$. 
Then $\mathcal{F}$ is indecomposable injective if and only if $\mathcal{F} \cong \sky(x)$ for some $x \in X$. 
\end{theorem}

\begin{proof}
By Proposition \ref{Pro:IndInjShIsGeneratedByAGlobalSec}, there exists a non-zero global section $s \in \mathcal{F}(X)$ such that $\mathcal{F} = \langle s \rangle$. 
The support of $\mathcal{F}$ is irreducible and closed by Propositions \ref{Prop:IndInjSuppIrr} and \ref{Prop:IndInjSuppClosed}. 
Since $X$ is sober, $\Supp(\mathcal{F}) = \supp(s) = \overline{ \{ x \}}$ for some $x \in X$. 
Hence, we obtain an isomorphism  
\[
  \mathcal{F} = \langle s \rangle \cong k_{\supp(s)} = k_{\overline{ \{ x \} } } = \sky(x)
\]
for some $x \in X$. 

The converse was shown in Lemma \ref{Lemma:IndicatorSheafFlabbyForClosedIrred}. 
\end{proof}

\begin{corollary}\label{Cor:IndInjAreIrredClosed}
Let $X$ be a topological space. 
There is a one-to-one correspondence with the indecomposable injective sheaves on $X$ and the points of the sobrification of $X$. 
\end{corollary}

\begin{proof}
By definition, $\Sob(X)$ is sober topological space, so by the previous theorem, there is a one-to-one correspondence with the points of $\Sob(X)$  
and indecoposable injective sheaves on $\Sob(X)$. 
But recall from Remark \ref{Rmk:SheavesOnSobEqv} that there is an equivalence of categories, $\Sh(X) \simeq \Sh(\Sob(X))$. 
Thus, the claim follows. 
\end{proof}

\begin{remark}\label{Rmk:Hoppner}
Corollary \ref{Cor:IndInjAreIrredClosed} may be viewed as a partial generalization of Höppner's classification of indecomposable injective persistence modules \cite[Proposition 1.1]{Hoppner}.
Let $P$ be a poset. 
Höppner showed that the indecomposable injective persistence modules over $P$ are in one-to-one correspondence with the upward-directed down-sets $D \subseteq P$.
This connection can be made more explicit by expressing these order-theoretic properties in topological terms. 
Recall that the category of persistence modules over $P$ is equivalent to $\Sh(P^a)$, and that a subset $D \subseteq P$ is upward-directed (resp.\ a down-set) if and only if it is irreducible (resp.\ closed) with respect to the Alexandrov topology.
\end{remark}

%%%%%%%%%%%%%%%%%%%%%%%%%%%%%%%%%%%%%%%%%%%%%%%%%
\subsection{Localizing Subcategories of Sheaves via Support Theory}\label{Subsec:SupportTheory}

For an abstract treatment of support theory, see \cite[Section 2]{supportTheoryYu}.
Let $X$ be a topological space, and let $\mathcal{F}$ be a sheaf on $X$. 
Recall that we write $\Supp(\mathcal{F})$ for the support of $\mathcal{F}$. 
It satisfies the following properties: 
\begin{enumerate}[(i)]
    \item $\Supp(\mathcal{F}) = \emptyset$ if and only if $\mathcal{F} = 0$; 
    \item $\Supp(\mathcal{F}) = \Supp(\mathcal{F}') \cup \Supp(\mathcal{F}'')$ for all short exact sequences $0 \to \mathcal{F}' \to \mathcal{F} \to \mathcal{F}'' \to 0$; 
    \item $\Supp \left( \bigoplus_{i \in I} \mathcal{F}_i \right) = \bigcup_{i \in I} \Supp(\mathcal{F}_i)$ for all sheaves $\mathcal{F}_i$ ($i \in I$). 
\end{enumerate}
In terms of \cite{supportTheoryYu}, $\Supp$ is a \emph{strong non-degenerate support}. 
The support induces two natural maps: 
\begin{itemize}
    \item For a localizing subcategory $\mathcal{L}$ of $\Sh(X)$, 
    \[
      \Supp(\mathcal{L}) 
      = \bigcup_{\mathcal{F} \in \mathcal{L}} \Supp(\mathcal{F}) 
      = \{ x \in X \mid \mathcal{F}_x \not= 0 \text{ for some } \mathcal{F} \in \mathcal{L} \};
    \]
    \item For a subset $Z \subseteq X$, we write $\Supp^{-1}(Z) = \{ \mathcal{F} \in \Sh(X) \mid \Supp(\mathcal{F}) \subseteq Z \}$. 
\end{itemize}
Note that the properties (i) -- (iii) imply that $\Supp^{-1}(Z)$ is a localizing subcategory of $\Sh(X)$ for all subsets $Z \subseteq X$. 
These two functions define a Galois connection 
\[
  \left\{ \text{Localizing  subcategories of $\Sh(X)$} \right\} 
  \xrightleftharpoons[\Supp^{-1}]{\Supp}  
  \left\{ \text{Subsets of $X$} \right\} 
\]
(see \cite[Remark 2.4]{supportTheoryYu}). 
In particular, we have $\mathcal{L} \subseteq (\Supp^{-1} \circ \Supp)(\mathcal{L})$ for all localizing subcategories $\mathcal{L}$ of $\Sh(X)$, and $(\Supp \circ \Supp^{-1})(Z) \subseteq Z$ for all subsets $Z \subseteq X$. 

The first step in classifying the localizing subcategories of $\Sh(X)$ is to force the support map to be surjective by restricting its codomain. 
In \cite{supportTheoryYu}, the notion of a supp-subset was introduced for this purpose. 
We, on the other hand, do not need this terminology, as we will see that the collection of the so-called Skula-open subsets will be the right choice for the codomain. 

Let $X$ be a topological space. 
Then the \emph{Skula topology} on $X$ is the topology generated by the open and closed sets of $X$. 
Equivalently, it is the topology on $X$ with a basis of locally closed sets of the original topology. 
We denote by $\Sk(X)$ the set $X$ equipped with the Skula topology. 
The open subsets of $\Sk(X)$ are called the \emph{Skula-open} subsets of $X$.

\begin{proposition}\label{Prop:OneSideInverseForSuppSets}
Let $X$ be a topological space, and let $Z \subseteq X$ be Skula-open. 
Then 
\[
  Z = (\Supp \circ \Supp^{-1})(Z). 
\]
\end{proposition}

\begin{proof}
First, recall that $(\Supp \circ \Supp^{-1})(Z) \subseteq Z$ for every subset $Z \subseteq X$. 
Let $x \in Z$. 
Since $Z$ is Skula-open, there exists a locally closed subset $Z_x \subseteq Z$ containing $x$. Consider the constant sheaf $k_{Z_x}$ on $Z_x$ and its extension by zero $i_!k_{Z_x}$ to $X$, where $i \colon Z_x \to X$ is the inclusion. 
By definition, $\Supp(i_! k_{Z_x}) = Z_x \subseteq Z$, so $i_! k_{Z_x} \in \Supp^{-1}(Z)$. 
Moreover, since $(i_! k_{Z_x})_x \not= 0$, we conclude that $x \in (\Supp \circ \Supp^{-1})(Z)$. 
\end{proof}

\begin{proposition}\label{Prop:SkulaOpenCorrespondsSupports}
Let $X$ be a topological space, and let $Z \subseteq X$. 
Then the following are equivalent: 
\begin{enumerate}[(i)]
    \item $Z$ is Skula-open; 
    \item $Z = \Supp(\mathcal{F})$ for some sheaf $\mathcal{F}$ on $X$; 
    \item $Z = \Supp(\mathcal{L})$ for some localizing subcategory $\mathcal{L}$ of $\Sh(X)$. 
\end{enumerate}
\end{proposition}

\begin{proof}
(ii) $\Rightarrow$ (i): 
Suppose that $Z = \Supp(\mathcal{F})$ for some sheaf $\mathcal{F}$ on $X$. 
We know that 
\[
  \Supp(\mathcal{F}) = \bigcup_{U}\bigcup_{s \in \mathcal{F}(U)} \supp(s).
\]
Therefore, $Z$ is Skula-open as a union of locally closed subsets.

(iii) $\Rightarrow$ (ii): 
Assume that $Z = \Supp(\mathcal{L})$ for some localizing subcategory $\mathcal{L}$ of $\Sh(X)$. For each $x \in Z$, choose a sheaf $\mathcal{F}(x) \in \mathcal{L}$ such that $\mathcal{F}(x)_x \not= 0$. 
Consider the direct sum $\mathcal{F} = \bigoplus_{x \in Z} \mathcal{F}(x)$. 
By construction, $\Supp(\mathcal{L}) \subseteq \bigcup_{x \in Z} \Supp(\mathcal{F}(x)) = \Supp(\mathcal{F})$. 
Conversely, as a localizing subcategory, $\mathcal{L}$ is closed under direct sums, so $\mathcal{F} \in \mathcal{L}$. 
In particular, $\Supp(\mathcal{F}) \subseteq \Supp(\mathcal{L})$. 
Thus, $\Supp(\mathcal{L}) = \Supp(\mathcal{F})$. 

(i) $\Rightarrow$ (iii): 
Finally, suppose that $Z$ is Skula-open. 
By Proposition \ref{Prop:OneSideInverseForSuppSets}, $Z = \Supp(\Supp^{-1}(Z))$. 
Hence, we may take $\mathcal{L} = \Supp^{-1}(Z)$.
\end{proof}

\begin{remark}\label{Rmk:SkulaOpenByLocSubCat}
Let $X$ be a topological space. 
We now have a Galois connection, where $\Supp$ is surjective: 
\[
  \{ \text{Localizing subcategories of } \Sh(X) \} \xrightleftharpoons[\Supp^{-1}]{\Supp} \{ \text{Skula-open subsets of } X \}. 
\]
However, it is not always a bijection. 
Consider, for example, the space $X = \Rset$ with the Alexandrov topology. 
Then the full subcategory of ephemeral modules $\Eph$ is a (bi)localizing subcategory of $\Sh(\Rset)$ (see Section \ref{Sec:ApplicationsToTDA}). 
The sheaf supported by a singleton $k_x$ is in $\Eph$ for all $x \in \Rset$. 
Thus, $\Supp(\Eph) = \Rset$. 
Obviously, the whole $\Sh(\Rset)$ is a localizing subcategory of itself, and $\Supp(\Sh(\Rset)) = \Rset$. 
Therefore, $\Supp$ cannot be an injection in the case of $X = \Rset$. 
For more details, see \cite{EphModAndScottShOverContPoset}. 
\end{remark}

%%%%%%%%%%%%%%%%%%%%%%%%%%%%%%%%%%%%%%%%%%%%%%%%%%%%%%%%%%%%
\section{Spectrums of Sheaves}

Let $X$ be a topological space. 
For brevity, we write $\InjSpec(X)$, $\Sp(X)$ and $\ASpec(X)$ for the injective spectrum, Gabriel spectrum and atom spectrum of the category of sheaves $\Sh(X)$, respectively,

%%%%%%%%%%%%%%%%%%%%%%%%%%%%%%%%%%%%%%%%%%%%%%%%%%%%
\subsection{Maps Between Spectra}

Let $X$ be a topological space. 
We denote by $\sky \colon X \to \InjSpec(X)$ the function, which maps a point to the corresponding skyscraper sheaf. 
In other words, $x \mapsto \sky(x)$ for all $x \in X$.

\begin{proposition}\label{Prop:SkyContMap}
Let $X$ be a topological space. 
Then 
\begin{enumerate}[(i)]
    \item 
    $\sky^{-1}(\mathcal{O}(\mathcal{L})) = \Supp(\mathcal{L})$ for all localizing subcategories $\mathcal{L}$ of $\Sh(X)$; 

    \item 
    $\sky^{-1}(\mathcal{O}(\Supp^{-1}(Z))) = Z$ for all Skula-open subsets $Z \subseteq X$; 

    \item 
    $\sky \colon \Sk(X) \to \Sp(X)$ is continuous. 
\end{enumerate}
\end{proposition}

\begin{proof}
(i) 
By Proposition \ref{Prop:SheafHomSuppSky}, 
\begin{align*}
    x \in \sky^{-1}(\mathcal{O}(\mathcal{L})) 
    &\iff \sky(x) \in \mathcal{O}(\mathcal{L}) \\
    &\iff \Hom(\mathcal{F}, \sky(x)) \not= 0 \text{ for some } \mathcal{F} \in \mathcal{L} \\
    &\iff x \in \Supp(\mathcal{F}) \text{ for some } \mathcal{F} \in \mathcal{L} \\
    &\iff x \in \Supp(\mathcal{L}). 
\end{align*}

(ii) 
This is a consequence of (i) and Proposition \ref{Prop:OneSideInverseForSuppSets}:
\[
  \sky^{-1}(\mathcal{O}(\Supp^{-1}(Z))) = \Supp(\Supp^{-1}(Z)) = Z. 
\]

(iii) 
This follows from (i) and Proposition \ref{Prop:SkulaOpenCorrespondsSupports}. 
\end{proof}

Recall that there is a well-defined map $\overline{E} \colon \ASpec(X) \to \InjSpec(X)$ such that $\overline{E}(\overline{H}) = E(H)$ for all monoform objects $H$ of $\Sh(X)$. 
By Lemma \ref{Lemma:ConstantSheafProperties}, $\sky(x)$ is monoform, and by Lemma \ref{Lemma:IndicatorSheafFlabbyForClosedIrred}, it is indecomposable injective. 
Thus, we have an atom $\overline{\sky(x)} \in \ASpec(X)$, and $\overline{E}(\overline{\sky(x)}) = E(\sky(x)) = \sky(x)$ for all $x \in X$.

\begin{proposition}\label{Prop:SuppAndASuppConnection}
Let $X$ be a topological space, and let $\mathcal{F}$ be a sheaf on $X$. 
Then the following are equivalent: 
\begin{enumerate}[(i)]
    \item $x \in \Supp(\mathcal{F})$; 

    \item $\sky(x) \in \mathcal{O}(\Supp^{-1}(\Supp(\mathcal{F})))$; 

    \item $\overline{\sky(x)} \in \ASupp(\mathcal{F})$. 
\end{enumerate}
\end{proposition}

\begin{proof}
(i) $\Leftrightarrow$ (ii) : 
By Proposition \ref{Prop:SkyContMap} (ii), 
\[
    \sky(x) \in \mathcal{O}(\Supp^{-1}(\Supp(\mathcal{F}))) 
    \iff x \in \sky^{-1}(\mathcal{O}(\Supp^{-1}(\Supp(\mathcal{F})))) = \Supp(\mathcal{F})
\]

(i) $\Leftrightarrow$ (iii) : 
Recall that $\ASupp(M) = \{ \overline{H} \in \ASpec(\mathcal{G}) \mid \Hom(M, E(H)) \not= 0 \}$ for all objects $M \in \mathcal{G}$. 
Thus, the claim follows from Proposition \ref{Prop:SheafHomSuppSky}: 
\[
  x \in \Supp(\mathcal{F}) 
  \iff \Hom(\mathcal{F}, \sky(x)) \not= 0 
  \iff \overline{\sky(x)} \in \ASupp(\mathcal{F}). 
\]
\end{proof}

%%%%%%%%%%%%%%%%%%%%%%%%%%%%%%%%%%%%%%%%%%%%%%%%%%%%%%%%%%
\subsection{Homeomorphism of the Gabriel Spectrum and the Atom Spectrum}

We first treat the case of sober spaces, beginning with an identification of the underlying sets of the atom spectrum and the injective spectrum.

\begin{proposition}\label{Prop:AtomToGabrielBijectionSober}
Let $X$ be a sober topological space. 
Then the map $\overline{E} \colon \ASpec(X) \to \InjSpec(X)$ is a bijection. 
\end{proposition}

\begin{proof}
The map is always an injection. 
By Theorem \ref{Thm:IndInjAreSkysrc}, $\InjSpec(X)$ consists of the skyscraper sheaves $\sky(x)$, where $x \in X$. 
The skyscraper sheaf $\sky(x)$ is monoform (Lemma \ref{Lemma:IndicatorSheafFlabbyForClosedIrred}) and injective, so $\overline{E}(\overline{\sky(x)}) = E(\sky(x)) = \sky(x)$ proving the surjectivity. 
\end{proof}

\begin{proposition}\label{Prop:AtomToGabrielContSober}
Let $X$ be a sober topological space, and let $\mathcal{L}$ be a localizing subcategory of $\Sh(X)$. Then $\overline{E}^{-1}(\mathcal O(\mathcal L))$ is atom-open. More precisely,
$\overline{E}^{-1}(\mathcal{O}(\mathcal{L})) = \ASupp(\mathcal{F})$ for some $\mathcal{F} \in \mathcal{L}$. 
In other words, $\overline{E} \colon \ASpec(X) \to \Sp(X)$ is a continuous map. 
\end{proposition}

\begin{proof}
By Proposition \ref{Prop:SkulaOpenCorrespondsSupports},  
$\Supp(\mathcal{L}) = \Supp(\mathcal{F})$ for some sheaf $\mathcal{F}$ on $X$. 
Now, by Propositions \ref{Prop:SkyContMap} (i) and \ref{Prop:SuppAndASuppConnection}, 
\begin{align*}
  \overline{\sky(x)} \in \overline{E}^{-1}(\mathcal{O}(\mathcal{L})) 
  &\iff \sky(x) \in \mathcal{O}(\mathcal{L}) \\
  &\iff x \in \Supp(\mathcal{L}) \\
  &\iff x \in \Supp(\mathcal{F}) \\
  &\iff \overline{\sky(x)} \in \ASupp(\mathcal{F}). 
\end{align*}
\end{proof}

It remains to show that the map is open.

\begin{proposition}\label{Prop:EIsOpenMapCaseSober}
Let $X$ be a sober topological space, and let $\mathcal{F}$ be a sheaf on $X$. 
Then $\overline{E}(\ASupp(\mathcal{F})) = \mathcal{O}(\Supp^{-1}(\Supp(\mathcal{F})))$. 
In particular, $\overline{E} \colon \ASpec(X) \to \Sp(X)$ is an open map. 
\end{proposition}

\begin{proof}
By Proposition \ref{Prop:SuppAndASuppConnection}, $\overline{E}^{-1}(\mathcal{O}(\Supp^{-1}(\Supp(\mathcal{F})))) = \ASupp(\mathcal{F})$. 
Since $\overline{E}$ is a bijection by Proposition \ref{Prop:AtomToGabrielBijectionSober}, we also have $\overline{E}(\ASupp(\mathcal{F})) = \mathcal{O}(\Supp^{-1}(\Supp(\mathcal{F})))$. 
\end{proof}

We now consider arbitrary topological spaces.

\begin{theorem}\label{Thm:AtomSpecHomeoGabriel}
Let $X$ be a topological space. 
Then the atom spectrum and the Gabriel spectrum are homeomorphic: 
\[
  \ASpec(X) \approx \Sp(X). 
\]
\end{theorem}

\begin{proof}
Suppose first that $X$ is sober. 
By Proposition~\ref{Prop:AtomToGabrielBijectionSober}, the map 
$\overline{E} \colon \ASpec(X) \to \Sp(X)$ is a bijection. 
Moreover, it is continuous by Proposition~\ref{Prop:AtomToGabrielContSober} and open by Proposition~\ref{Prop:EIsOpenMapCaseSober}. Hence \(\overline{E}\) is a homeomorphism.

Now let $X$ be arbitrary. By Remark \ref{Rmk:SheavesOnSobEqv}, there is an equivalence of  categories $\Sh(X) \simeq \Sh(\Sob(X))$. Therefore
\[
  \ASpec(X) \approx \ASpec(\Sob(X)) \approx \Sp(\Sob(X)) \approx \Sp(X),
\]
where the middle homeomorphism follows from the sober case.
\end{proof}

%%%%%%%%%%%%%%%%%%%%%%%%%%%%%%%%%%%%%%%%%%%%%%%%%%%%%%%
\subsection{Gabriel Spectrum of Sheaves}

We are now going to prove the Main Theorem (Corollary \ref{Cor:GabrielIsSkulaSob}). We observe first the following:

\begin{proposition}\label{Prop:SkyInjBij}
Let $X$ be a topological space. 
\begin{enumerate}[(i)]
    \item 
    If $X$ is $T_0$, then $\sky \colon X \to \InjSpec(X)$ an injection; 

    \item 
    If $X$ is sober, then $\sky \colon X \to \InjSpec(X)$ a bijection. 
\end{enumerate}
\end{proposition}

\begin{proof}
(i) 
Recall that in a $T_0$-space, we have an equivalence: 
\[
  \overline{\{ x \}} = \overline{\{ y \} } \iff x = y. 
\]
Thus, $\sky(x) = \sky(y)$ if and only if $x = y$. 

(ii) 
Every sober space is $T_0$, so the injectivity follows from (i). 
The surjectivity follows from Theorem \ref{Thm:IndInjAreSkysrc}. 
\end{proof}

\begin{theorem}\label{Thm:ForSobSpaceGabrielIsSkula}
Let $X$ be a sober topological space. 
Then there exist a homeomorphism for the Gabriel spectrum: 
\[
   \Sp(X) \approx \Sk(X). 
\]
\end{theorem}

\begin{proof}
Since the map $\sky \colon \Sk(X) \to \Sp(X)$ is a bijection by Proposition \ref{Prop:SkyInjBij} (ii), it follows from Proposition \ref{Prop:SkyContMap} that $\sky(\Supp(\mathcal{L})) = \mathcal{O}(\mathcal{L})$ for all localizing subcategories $\mathcal{L}$ of $\Sh(X)$. 
In other words, $\sky$ is an open map, in addition being a continuous bijection. 
Thus, it is a homeomorphism. 
\end{proof}

\begin{corollary}\label{Cor:GabrielIsSkulaSob}
Let $X$ be a topological space. 
Then there exists a homeomorphism for the Gabriel spectrum: 
\[
   \Sp(X) \approx \Sk(\Sob(X)). 
\]
\end{corollary}

\begin{proof}
Note that there is an equivalence of sheaf categories $\Sh(X) \simeq \Sh(\Sob(X))$ (see Remark \ref{Rmk:SheavesOnSobEqv}). 
Thus, 
\[
  \Sp(X) 
  \approx \Sp(\Sob(X)) 
  \approx \Sk(\Sob(X)). 
\]
\end{proof}

%%%%%%%%%%%%%%%%%%%%%%%%%%%%%%%%%%%%%%%%%%%%%%%%%%%%%%%%%%%%%%%%%%%%%%%%
\subsection{Topological Properties of the Gabriel Spectrum}

In consequence of Corollary \ref{Cor:GabrielIsSkulaSob}, we can recognize plenty of topological properties of the Gabriel spectrum.

\begin{theorem}\label{Thm:GabrielSpectrumTopProperties}
Let $X$ be a topological space. 
Then the Gabriel spectrum $\Sp(X)$ 
\begin{enumerate}[(i)]
    \item is Hausdorff; 
    \item admits a basis of clopen subsets; 
    \item is zero-dimensional with respect to the small inductive dimension; 
    \item is totally disconnected. 
\end{enumerate}
\end{theorem}

\begin{proof}
(i) 
It is known that the Skula topology of a $T_0$-space is Hausdorff \cite[Exercise 9.7.16]{goubault-larrecq_2013}. 
Since the sobrification $\Sob(X)$ is always a $T_0$-space, it follows that $\Sk(\Sob(X))$ is Hausdorff. 
Corollary~\ref{Cor:GabrielIsSkulaSob} therefore implies that $\Sp(X)$ is Hausdorff.

(ii)
The Skula topology admits a basis of clopen subsets \cite[p.~146]{sobrificationHoffmannRemainder}. Hence the claim follows from Corollary~\ref{Cor:GabrielIsSkulaSob}.

(iii) and (iv) 
Both statements are immediate consequences of (ii).
\end{proof}

\begin{theorem}\label{Thm:ProductOfGabrielSpectrums}
Let $X$ and $Y$ be topological spaces. 
Then $\Sp(X \times Y) \approx \Sp(X) \times \Sp(Y)$. 
\end{theorem}

\begin{proof}
The Skula topology respects the product topology: $\Sk(X \times Y) = \Sk(X) \times \Sk(Y)$ (\cite[Lemma 1.2]{sobrificationHoffmannRemainder}). 
Now, by Remark \ref{Rmk:SobPresProd} and Corollary \ref{Cor:GabrielIsSkulaSob}, 
\[
  \Sp(X \times Y) 
  \approx \Sk(\Sob(X \times Y)) 
  = \Sk(\Sob(X)) \times \Sk(\Sob(Y)) 
  \approx \Sp(X) \times  \Sp(Y). 
\]
\end{proof}

\begin{theorem}\label{Thm:GabrielSpecCompact}
Let $X$ be a topological space. 
Then the Gabriel spectrum $\Sp(X)$ is compact if and only if $X$ is a Noetherian space. 
\end{theorem}

\begin{proof}
Hoffmann proved that $\Sk(Y)$ is compact if and only if $Y$ is sober and Noetherian \cite[Theorem~3.1]{sobrificationHoffmannRemainder}. By Corollary~\ref{Cor:GabrielIsSkulaSob}, there is a homeomorphism $\Sp(X)\cong \Sk(\Sob(X))$. Since $\Sob(X)$ is always sober, it follows that $\Sp(X)$ is compact if and only if $\Sob(X)$ is Noetherian. Finally, a topological space is Noetherian if and only if its sobrification is Noetherian. Hence $\Sp(X)$ is compact if and only if $X$ is Noetherian. 
\end{proof}

\begin{theorem}\label{Thm:GabrielForHausdorff}
Let $X$ be a Hausdorff space. 
Then the Gabriel spectrum $\Sp(X)$ is discrete (with cardinality $\abs{X}$). 
\end{theorem}

\begin{proof}
Since $X$ is Hausdorff, it is sober. Hence, by Theorem \ref{Thm:ForSobSpaceGabrielIsSkula}, $\Sp(X) \cong \Sk(X)$.
Moreover, every singleton of $X$ is closed, and therefore Skula-open. It follows that every singleton is open in $\Sk(X)$, so $\Sk(X)$ is discrete. Consequently, $\Sp(X)$ is also discrete.
\end{proof}

%%%%%%%%%%%%%%%%%%%%%%%%%%%%%%%%%%%%%%%%%%%%%%%%%%%%%%%
\section{Applications to Topological Data Analysis}\label{Sec:ApplicationsToTDA}

%%%%%%%%%%%%%%%%%%%%%%%%%%%%%%%%%%%%%%%%%%%%%%%%%%%%%%%%%%%%%%
\subsection{Persistence Modules and Scott Sheaves}

Let $P$ be a poset, let $k$ be a field, and let $\Vect$ denote the category of $k$-vector spaces. If $\posetC$ is the thin category associated with $P$, then a (covariant) functor $M\colon \posetC\to\Vect$ is called a \emph{persistence module} over $P$. Thus, for every $p\in P$, there is a $k$-vector space $M_p$, and for every relation $p\le q$, there is a $k$-linear map $M(p\le q)\colon M_p\to M_q$, called an \emph{internal morphism}.
These morphisms satisfy 
\[
  M(p \le p) = \id \quad \text{and} \quad M(q \le r) \circ M(p \le q) = M(p \le r)
\]
for all $p \le q \le r$ in $P$. 
The persistence modules form a category, where the morphisms are just natural transformations, and it is denoted by $\Fun(\posetC, \Vect)$. 
Let $I \subseteq P$ be a convex set. 
The \emph{indicator module} $k[I]$ is the persistence module defined by setting 
\[
  k[I]_p = 
    \begin{cases}
      k, &\text{if } p \in I; \\
      0, &\text{otherwise}. 
    \end{cases}
\]
for all $p\in P$.
The internal morphism $k[I](p \le q)$ is the identity if $p,q \in I$, and zero otherwise. 
When $I = \{ p \}$ for some $p \in P$, we write $k_p$ for the corresponding indicator module. 

For a given poset $P$, the collection of the basic up-sets $U_p = \{ x \in P \mid x \ge p \}$, where $p \in P$, form a basis of the so-called \emph{Alexandrov topology}. 
A subset of $P$ is open with respect to the Alexandrov topology if and only if it is an up-set. 
We write $P^a$ for the poset $P$ endowed with the Alexandrov topology. 
It is well known that the category of persistence modules over $P$ is equivalent to the category of sheaves on $P^a$ \cite[Theorem~4.2.10]{Curry1}. 
In \cite{EphModAndScottShOverContPoset}, it was observed that the so-called 
Scott topology also plays a natural role in topological data analysis. 
Since the Scott topology originates in domain theory, we briefly recall some relevant notions from domain theory. For further background, see \cite{goubault-larrecq_2013}.

Let $P$ be a poset.
For $x,y \in P$, we say that $x$ is \emph{way below} $y$ (and $y$ is \emph{way above} $x$), and write $x \ll y$, if for all directed sets $D \subseteq P$ with $y \le \sup D$ (whenever the supremum exists) there exists $d \in D$ such that $x \le d$. 
An element $x \in P$ is called \emph{compact} (or \emph{isolated from below}) if $x \ll x$. 
We also write ${\twoheaddownarrow} p = \{ x \in P \mid x \ll p \}$ for all $p \in P$. 
A poset is called a \emph{directed-complete partial order}, or a \emph{dcpo} for short, if each of its directed subsets has a supremum. 
A poset $P$ is \emph{continuous}, if for all $p \in P$ the set ${\twoheaddownarrow} p$ is
upward-directed with supremum $p$.

A subset $U \subseteq P$ is called \emph{Scott-open}, if it is an up-set and for every directed subset $D \subseteq P$, the intersection $D \cap U$ is non-empty whenever $\sup D$ exists and is in $U$. The Scott-open subsets form a topology, called the \emph{Scott topology}. We write $P^\sigma$ for the poset $P$ endowed with this topology.
Since every Scott-open subset is an up-set, the identity map $j \colon P^a \to P^{\sigma}$ is continuous.
If the poset $P$ is a continuous dcpo, then $P^{\sigma}$ is sober \cite[Proposition 8.2.12 (b)]{goubault-larrecq_2013}.

A subset $F \subseteq P$ is called an \emph{ideal} if it is downward-closed and upward-directed. For $p \in P$, the basic down-set $D_p = \{ x \in P \mid x \le p \}$ is the simplest example of an ideal. The set of all ideals of $P$ is denoted by $\Idl(P)$ and is ordered by inclusion.

\begin{proposition}\label{Prop:BasisOfSkulaOfContPoset}
Let $P$ be a continuous poset. 
Then the Skula topology of $P^{\sigma}$ has a basis consisting of sets $\Int U_x \cap D_y$, where $x,y \in P$. 
In particular, if $P = \Rset^n$, this basis consists of the rectangles \[ ]x_1, y_1] \times \dots \times \left] x_n, y_n \right]\] where
$x_i, y_i \in \Rset$ for all $i = 1, \dots, n$.
\end{proposition}

\begin{proof} 
Recall that the collection of sets $\Int U_x$ $(x\in P)$ forms a basis for the Scott topology. Moreover, every down-set $D\subseteq P$ can be written as a union of basic down-sets, 
\[ 
  D=\bigcup_{y\in D} D_y, 
\] 
and each $D_y$ is Scott-closed. 
Since the Skula topology is generated by the Scott-open subsets together with the Scott-closed subsets, it follows that $\Sk(P^\sigma)$ admits a basis consisting of the sets $\Int U_x \cap D_y$, where $x,y\in P$. 
For $\bm{x}=(x_1,\dots,x_n)$ and $\bm{y}=(y_1,\dots,y_n)$ in $\Rset^n$, we have 
\[
  \Int U_{\bm{x}} = \left]x_1,\infty\right[\times\cdots\times\left]x_n,\infty\right[ 
  \quad \text{and} \quad  
  D_{\bm{y}} = \left]-\infty,y_1\right]\times\cdots\times\left]-\infty,y_n\right]. 
\] 
Therefore, 
\[ 
  \Int U_{\bm{x}}\cap D_{\bm{y}} = \left]x_1,y_1\right]\times\cdots\times\left]x_n,y_n\right], 
\] 
which proves the second claim. 
\end{proof}

\begin{remark}
Note that the Sorgenfrey topology on $\Rset$ is often defined as the topology generated by the half-open intervals $[x,y[$, where $x,y \in \Rset$. 
This \emph{Sorgenfrey line} is clearly homeomorphic to the space considered in the previous proposition for $n = 1$. 
In this article, we use the former definition for the Sorgenfrey topology instead, since it is the one related to the Scott topology on $\Rset$.
\end{remark}

Let $P$ be a continuous poset. 
The sheaves on $P^{\sigma}$ are called \emph{Scott sheaves}, for short. 
In \cite[Theorem 3.18]{EphModAndScottShOverContPoset} we observe that the category of Scott sheaves is equivalent with the full subcategories of upper and lower 
semi-continuous persistence modules, making 
Scott sheaves interesting for topological data analysis. 
For a persistence module $M$ over $P$, we set
\[
  \underline{M}_p = \varprojlim_{x \gg p} M_x \quad \text{and} \quad \overline{M}_p = \varinjlim_{x \ll p} M_x 
\]
for all $p \in P$. 
A persistence module $M$ is called \emph{upper semi-continuous}, 
if the canonical morphism $M \to \underline{M}$ is an isomorphism. 
Dually, $M$ is \emph{lower semi-continuous}, if the canonical morphism $\overline{M} \to M$ is an isomorphism. 
We denote the full subcategories of upper semi-continuous modules and lower semi-continuous modules by 
$\Fun^c(\posetC, \Mod)$ and $\Fun_c(\posetC, \Mod)$, respectively. 

A persistence module $M$ over a continuous poset is called \emph{ephemeral} if $M(p \le q)=0$ for all $p\ll q$. 
The full subcategory of ephemeral modules is denoted by $\Eph$, and it is known to be a (bi)localizing subcategory of $\Fun(\posetC, \Vect)$. 
Moreover, the main result of \cite[Theorem 4.17]{EphModAndScottShOverContPoset} states that the quotient category $\Fun(\posetC,\Vect) / \Eph$ is equivalent to the category of Scott sheaves $\Sh(P^{\sigma})$. 
For further details on Scott sheaves, ephemeral modules, and semi-continuous modules, see \cite{EphModAndScottShOverContPoset}.

Next, we consider the Gabriel spectra of persistence modules and Scott sheaves. 
The poset $\Rset^n$ provides a particularly useful example for illustrating these spectra. 
Note that a subset of a poset $P$ is an ideal if and only if it is closed and irreducible with respect to the Alexandrov topology. 
Therefore, the underlying set of the sobrification $\Sob(P^a)$ may be identified with the set of ideals $\Idl(P)$.
In the case $P = \Rset$, there are three types of ideals: $\Rset$ itself, the closed ideals $\left]-\infty,x\right]$ and the open ideals $\left]-\infty,x\right[$ where $x \in \Rset$. 
Since the sobrification commutes with products (see Remark \ref{Rmk:SobPresProd}), an ideal of $\Rset^n$, for $n \ge 1$, can be expressed as a product of these ideals. 
In other words, 
\[
  \Idl(\Rset^n) = \Idl(\Rset)^n. 
\]

We begin with the persistence modules over $\Rset^n$, where the spectrum is now naturally described in terms of a topological space of ideals.

\begin{theorem}\label{Thm:GabrielSpecOfPerModOverR}
Let $P = \Rset^n$. Then there is a homeomorphism  
\[
  \Sp(\Fun(\Rset^n, \Vect)) \approx \Idl(\Rset^n)
\]
where $\Idl(\Rset^n)$ is endowed with the topology generated by the rectangles \[\left] F_1, F_1' \right] \times \dots \times \left] F_{n}, F_{n}' \right],\] for $F_i,F_i' \in \Idl(\Rset)$ ($i = 1, \dots, {n}$). 
\end{theorem}

\begin{proof}
We identify persistence modules over $\Rset^n$ with sheaves on $(\Rset^n)^a$. 

First consider the case $n = 1$.
The proper non-empty open subsets of $\Rset^a$ are precisely $\left[x, \infty \right[$ and $\left] x, \infty \right[$, for $x \in \Rset$. 
The corresponding open subsets in the sobrification $\Sob(\Rset^a)$ are 
\[
  \diamond \left[x, \infty \right[ 
  = \{ F \in \Idl(\Rset) \mid \left]-\infty, x \right] \subseteq F \} 
  = \{ F \in \Idl(\Rset) \mid \left]-\infty, x \right[ \subsetneq F \} 
\]
and 
\[
  \diamond \left]x, \infty \right[ 
  = \{ F \in \Idl(\Rset) \mid \left]-\infty, x \right] \subsetneq F \}. 
\]
The closed subsets are their complements, namely 
\[
  \{ F \in \Idl(\Rset) \mid F \subseteq \left] -\infty, x \right[ \}
  \quad \text{and} \quad 
  \{ F \in \Idl(\Rset) \mid F \subseteq \left] -\infty, x \right] \}. 
\]
The Skula topology of $\Sob(\Rset^a)$ is therefore generated by the sets 
\[
  \left] F_1, F_2 \right] = \{ F \in \Idl(\Rset) \mid F_1 \subsetneq F \subseteq F_2 \}, 
\]
where $F_1, F_2 \in \Idl(\Rset)$. 
The claim now follows from Corollary \ref{Cor:GabrielIsSkulaSob}. 

We now turn to the general case.
Recall that the Alexandrov topology preserves products. 
In particular, $(\Rset^n)^a = (\Rset^a)^n$ for all $n \ge 1$. 
Therefore, the case $n=1$, together with Theorem \ref{Thm:ProductOfGabrielSpectrums}, yields 
\[
  \Sp(\Fun(\Rset^n, \Vect)) \approx \Sp(\Fun(\Rset, \Vect))^n \approx \Idl(\Rset)^n = \Idl(\Rset^n), 
\]
where $\Idl(\Rset^n)$ is endowed with corresponding product topology. 
This topology has a basis that consists of the rectangles $\left] F_1, F_1' \right] \times \dots \times \left] F_{n}, F_{n}' \right]$, where $F_i,F_i' \in \Idl(\Rset)$ for $i = 1, \dots, n$. 
\end{proof}

For Scott sheaves, the situation is even more concrete. The points of the spectrum are parametrized by $(\Rset\cup\{\infty\})^n$, and the Gabriel topology becomes the product Sorgenfrey topology.

\begin{theorem}\label{Thm:GabrielSpecOfScottShOverR}
Let $P = \Rset^{n}$. 
Then there is a homeomorphism  
\[
  \Sp((\Rset^n)^{\sigma}) \approx (\Rset \cup \{ \infty \})^n, 
\]
where $(\Rset \cup \{ \infty \})^n$ is endowed with the product Sorgenfrey topology, that is, the topology generated by the rectangles \[ ]x_1,y_1]\times\cdots\times ]x_n,y_n], \] with $x_i\in\Rset$ and $y_i\in\Rset\cup\{\infty\}$.
\end{theorem}

\begin{proof}
Let us first consider the case $n = 1$. 
The sobrification of $\Rset^{\sigma}$ is $(\Rset \cup \{ \infty \})^{\sigma}$ as the smallest sober space containing $\Rset^{\sigma}$. 
By Proposition \ref{Prop:BasisOfSkulaOfContPoset}, its Skula topology is generated by the intervals $\left] x, y \right]$, for $x < y$.
Corollary~\ref{Cor:GabrielIsSkulaSob} therefore yields $\Sp(\Rset^\sigma) \approx \Rset\cup\{\infty\}$ endowed with the Sorgenfrey topology.

For the general case, recall that the Scott topology preserves products of continuous posets \cite[Proposition 5.1.54]{goubault-larrecq_2013}. 
Hence, $(\Rset^n)^{\sigma} = (\Rset^{\sigma})^n$ for all $n \ge 1$. 
Thus, by Theorem \ref{Thm:ProductOfGabrielSpectrums}, 
\[
  \Sp((\Rset^n)^{\sigma}) \approx \Sp(\Rset^{\sigma})^n \approx (\Rset \cup \{ \infty \})^n, 
\]
where the right-hand side is endowed with the product Sorgenfrey topology. Equivalently, its topology is generated by the boxes \[ ]x_1,y_1]\times\cdots\times ]x_n,y_n]. \]  
\end{proof}

\begin{remark}\label{Rmk:Embedding}
Let $\mathcal{G}$ be a Grothendieck category. 
If $\mathcal{L}$ is a localizing subcategory of $\mathcal{G}$, then both $\mathcal{L}$ and the quotient category $\mathcal{G} / \mathcal{L}$ are Grothendieck categories. 
An indecomposable injective object of $\mathcal L$ is uniform when regarded as an object of $\mathcal G$, and therefore its injective envelope in $\mathcal G$ is an indecomposable injective object of $\mathcal G$. This defines an injective map
$\InjSpec(\mathcal{L}) \to \InjSpec(\mathcal{G})$. 
The image of this map is precisely the Gabriel open set $\mathcal{O}(\mathcal{L})$. 
Consequently, $\Sp(\mathcal{L})$ is homeomorphic to an open subspace of $\Sp(\mathcal{G})$. 
Similarly, the section functor $\sec \colon \mathcal{G} / \mathcal{L} \to \mathcal{G}$ preserves indecomposable injective objects, and therefore induces an injective map $\InjSpec(\mathcal{G} / \mathcal{L}) \to \InjSpec(\mathcal{G})$. 
The image of this map is the complementary Gabriel closed subset $\Sp(\mathcal{G}) \setminus \mathcal{O}(\mathcal{L})$. 
For more details, see, for example, \cite[Corollary 2.2.18]{KrauseHomTheoryRep}. 

In the situation $\mathcal{G} = \Fun(\posetC, \Vect)$ and $\mathcal{L} = \Eph$, where $P$ is a continuous poset, it then follows that $\Sp(\Eph)$ is homeomorphic to an open subspace of $\Sp(\Fun(\posetC, \Vect))$, whereas $\Sp(P^{\sigma})$ is homeomorphic to the complementary closed subspace.
\end{remark}

When $P=\Rset$, the Gabriel spectrum of ephemeral modules has a particularly simple structure.

\begin{proposition}\label{Prop:GabrielSpecOfDenseLinPoset}
Let $P = \Rset$. Then $\Sp(\Eph)$ is homeomorphic to $\Rset$ endowed with discrete topology. 
\end{proposition}

\begin{proof} 
By Remark \ref{Rmk:Embedding}, $\Sp(\Eph)$ is homeomorphic to the open set $\mathcal{O}(\mathcal{\Eph})$ of $\Sp(\Fun(\Rset, \Vect))$. The injective spectrum of
$\Fun(\Rset, \Vect)$ consists of the interval modules $k\left] -\infty, p \right]$ and $k\left] -\infty, p \right[$, where $p\in \Rset$, together with the constant module $k[\Rset]$.
Recall that a persistence module $M\in \Fun(\Rset, \Vect)$ is ephemeral if and only if all its internal morphisms are zero. A direct calculation shows that \[\Hom(M, k\left] -\infty, p \right])=\Hom_k(M_p,k)\quad\text{and}\quad \Hom(M, k\left] -\infty, p \right[])=0.\] Moreover,  
$\Hom(M, k[\Rset]) = 0$. Therefore 
\[\mathcal{O}(\Eph) = \{ k\left] -\infty, p \right] \mid p \in \Rset \}\] 
Furthermore,
\[\mathcal{O}(\Supp^{-1}(\{p\})) = \{ k\left] -\infty, p \right] \}\] for all $p \in \Rset$.
Hence every point $k\left] -\infty, p \right]$ is open in $\Sp(\Fun(\Rset, \Vect))$.
Thus $\mathcal O(\Eph)$ is a discrete subspace. Since its points are naturally parametrized by $\Rset$, it follows that $\Sp(\Eph)$ is homeomorphic to $\Rset$ endowed with the discrete topology.
\end{proof}

We also obtain the following general results concerning Gabriel spectra of Scott sheaves.

\begin{theorem}
Let $P$ be a linearly ordered poset. 
Then the Gabriel spectrum $\Sp(P^{\sigma})$ is compact if and only if $P$ is well-ordered. 
\end{theorem}

\begin{proof}
Recall that a linearly ordered poset $P$ is well-ordered, or equivalently Artinian, if and only if $P^\sigma$ is a Noetherian space. The claim therefore follows from Theorem~\ref{Thm:GabrielSpecCompact}.
\end{proof}

\begin{theorem}
Let $P$ be a continuous poset. 
Then 
\begin{enumerate}[(i)]
    \item a singleton $\{ p \}$ $(p \in P)$ is Skula-open in the Scott topology if and only if $p$ is compact; 
    \item $P$ is Noetherian if and only if $P$ is a dcpo and $p$ is compact for all $p \in P$; 
    \item $P$ is Noetherian if and only if $\Sob(P^{\sigma})$ is Noetherian as a poset; 
    \item the Gabriel spectrum $\Sp(P^{\sigma})$ is discrete if and only if $P$ is a Noetherian poset. 
\end{enumerate}
\end{theorem}

\begin{proof}
(i) 
Simply, 
\[
  \{ p \} \text{ is Skula-open} 
  \iff U_p \text{ is Scott-open} 
  \iff p \ll p 
  \iff p \text{ is compact}.
\]

(ii)
See \cite[Remark 2.1]{li2026conoetherianspaces}. 

(iii)
Note that $P$ is a subposet of $\Sob(P^\sigma)$. Hence, if $\Sob(P^\sigma)$ is Noetherian, then so is $P$. Conversely, if $P$ is Noetherian, then $P$ is  dcpo by (ii). We then know by \cite[Proposition 8.2.12]{goubault-larrecq_2013} that $P^\sigma$ is sober, and therefore $\Sob(P^\sigma)=P$. Hence $\Sob(P^\sigma)$ is a Noetherian poset.

(iv) By \cite[Exercise 8.2.48]{goubault-larrecq_2013} the topological space $\Sob(P^{\sigma})$ is the Scott space of the underlying poset, which coincides with the so-called rounded ideal completion.  Moreover, we know by \cite{goubault-larrecq_2013} that this poset is always a dcpo. The claim now follows from Corollary \ref{Cor:GabrielIsSkulaSob}, (i), (ii) and (iii): 
\begin{align*}
  \Sp(P^{\sigma}) \text{ is discrete} 
  &\iff \Sk(\Sob(P^{\sigma})) \text{ is discrete} \\
  &\iff \text{every point of } \Sob(P^{\sigma}) \text{ is Skula-open} \\
  &\iff \text{every point of } \Sob(P^{\sigma}) \text{ is compact} \\
  &\iff \Sob(P^{\sigma}) \text{ is a Noetherian poset} \\
  &\iff P \text{ is a Noetherian poset}. 
\end{align*}
\end{proof}

%%%%%%%%%%%%%%%%%%%%%%%%%%%%%%%%%%%%%%%%%%%%%%%%%%%%
\subsection{Localizing Subcategories of Scott Sheaves on Reals}

It is well known that classifying localizing subcategories is generally a difficult problem. Nevertheless, we obtain a complete classification in one non-trivial case, namely when $ P=\Rset\cup\{\infty\}$ is endowed with the Scott topology.

\begin{proposition}
Let $X$ be a topological space, and let $\mathcal{L}$ be a localizing subcategory of $\Sh(X)$. 
Then $x \in \Supp(\mathcal{L})$ if and only if $k_Z \in \mathcal{L}$ for some locally closed set $Z$ containing $x$. 
\end{proposition}

\begin{proof}
If $x \in \Supp(\mathcal{L})$, then $x \in \Supp(\mathcal{F})$ for some $\mathcal{F} \in \mathcal{L}$. 
By Lemma \ref{Lemma:SectionInduceMonomorf} (ii), there exists a locally closed subset $Z \subseteq X$ and a monomorphism $k_Z \to \mathcal{F}$. 
Since $\mathcal{L}$ is closed under subobjects, it follows that $k_Z \in \mathcal{L}$.

Conversely, if $k_Z \in \mathcal{L}$ for some locally closed subset
$Z$ containing $x$, then $x \in \Supp(k_Z) \subseteq \Supp(\mathcal{L})$. 
\end{proof}

\begin{lemma}\label{Lemma:PointInSupportImpliesInterval2}
Let $P$ be a continuous poset and let $\mathcal{L}$ be a localizing subcategory of $\Sh(P^{\sigma})$. 
Then $p \in \Supp(\mathcal{L})$ if and only if $j_*k[x,p] \in \mathcal{L}$ for some $x \ll p$. 
\end{lemma}

\begin{proof}
Suppose first that $p \in \Supp(\mathcal{L})$. By the previous proposition,
$j_*k[I] \in \mathcal{L}$ for some locally Scott-closed interval
$I \subseteq P$ containing $p$.
In particular, $\Int U_x \cap D_p \subseteq I$ for some $x \ll p$.
Therefore, $j_*k[x,p]$ is a subquotient of $j_*k[I]$, and hence $j_*k[x,p] \in \mathcal{L}$.

Conversely, if $j_*k[x, p] \in \mathcal{L}$ for some $x \ll p$, then $p \in \Supp(\mathcal{L})$, because $j_*k[x,p]_p = k \not= 0$, 
so $p \in \Supp(j_*k[x,p]) \subseteq \Supp(\mathcal{L})$.
\end{proof}

\begin{proposition}\label{Prop:IndicatorsInLocalizingSubcategory3}
Let $P = \Rset \cup \{ \infty \}$, and let $\mathcal{L}$ be a localizing subcategory of $\Sh(P^{\sigma})$. 
\begin{enumerate}[(i)]
    \item If $p, q \in \Rset \cup \{ \infty \}$ such that $p \ll q$ and $\left] p, q \right] \subseteq \Supp(\mathcal{L})$, then $j_*k[p,q] \in \mathcal{L}$; 

    \item If $I \subseteq \Rset \cup \{ \infty \}$ is a locally Scott-closed interval such that $I \subseteq \Supp(\mathcal{L})$, then $j_*k[I] \in \mathcal{L}$.  
\end{enumerate}
\end{proposition}

\begin{proof}
(i) 
Consider the set
\[
  A = \{ x \in [p,q] \mid j_*k[x,q] \in \mathcal{L} \}.
\]
First, $A$ is non-empty, since $q \in A$ (indeed, $j_*k[q,q]=0\in\mathcal{L}$). Moreover, $A$ is bounded below. Hence, by the completeness of $\Rset$, $A$ has an infimum, which we denote by $a = \inf A$.

We first show that $j_*k[a,q] \in \mathcal{L}$. Since $a = \inf A$, every point $x$ satisfying $a < x \le q$ belongs to $A$. 
Thus, $j_*k[x,q] \in \mathcal{L}$ for every $a < x \le q$. 
These sheaves form a directed system $(j_*k[x,q])_{a < x \le q}$, and
\[
  j_*k[a,q] = \varinjlim_{a < x \le q} j_*k[x,q].
\]
Since $\mathcal{L}$ is closed under direct limits as a localizing subcategory, it follows that
$j_*k[a,q]\in\mathcal{L}$. 
In particular, $a \in A$. 

It remains to prove that $a = p$. 
Suppose, to the contrary, that $a > p$. 
Then $a \in \left] p,q \right] \subseteq \Supp(\mathcal{L})$.
By Lemma \ref{Lemma:PointInSupportImpliesInterval2}, there exists $y < a$ such that $j_*k[y,a] \in \mathcal{L}$.
Since
$\left] y, q \right] = \left] y, a \right] \cup \left] a , q \right]$, there is a short exact sequence
\[
  0 \to j_*k[a,q] \to j_*k[y, q] \to j_*k[y,a] \to 0. 
\]
Because $\mathcal{L}$ is closed under extensions, $j_*k[y,q]\in\mathcal{L}$. 
Hence $y \in A$, contradicting the fact that $a = \inf A$, since $y < a$.
Therefore $p = a \in A$, and consequently $j_*k[p,q]\in\mathcal{L}$.

(ii) 
By the previous item, the claim has already been established for all bounded locally Scott-closed intervals. Thus, it remains only to consider the unbounded intervals $I= \left] -\infty , q \right]$, where $q\in\Rset\cup\{\infty\}$.

For every $x\in I$, we have
\[
  \left] x,q \right] \subseteq I \subseteq \Supp(\mathcal L).
\]
Hence, (i) yields $j_*k[x,q]\in\mathcal L$.
Moreover, the sheaves $(j_*k[x,q])_{x\in I}$ form a directed system, and
\[
  j_*k[I] = \varinjlim_{x\in I}j_*k[x,q].
\]
Since $\mathcal L$ is closed under direct limits, it follows that $j_*k[I]\in\mathcal L$.
\end{proof}

We can now show that when $P = \Rset \cup \{ \infty \}$, a localizing subcategory is completely determined by its support.

\begin{theorem}\label{Thm:OneSideInverseForLocSubCat}
Let $P = \Rset \cup \{ \infty \}$, and let $\mathcal{L}$ be a localizing subcategory of $\Sh(P^{\sigma})$. 
Then 
\[
  \mathcal{L} = (\Supp^{-1} \circ \Supp)(\mathcal{L}). 
\]
\end{theorem}

\begin{proof}
Recall first that $\mathcal{L} \subseteq (\Supp^{-1} \circ \Supp)(\mathcal{L})$. 
Suppose now that $\mathcal{F} \in (\Supp^{-1} \circ \Supp)(\mathcal{L})$. 
Equivalently, $\Supp(\mathcal{F}) \subseteq \Supp(\mathcal{L})$. 
Fix $p \in \Supp(\mathcal{F})$ and a non-zero germ $s_p \in \mathcal{F}_p$. 
By Lemma \ref{Lemma:SectionInduceMonomorf}, there exist a locally Scott-closed interval $I_{s,p} \subseteq P$ that contains $p$ and a morphism $\varphi \colon j_*k[I_{s,p}] \to \mathcal{F}$, where $\varphi_p(1) = s_p \not= 0$. 
Thus, 
\[
  I_{s,p} = \supp(s) \subseteq \Supp(\mathcal{F}) \subseteq \Supp(\mathcal{L}). 
\]
It follows from Proposition \ref{Prop:IndicatorsInLocalizingSubcategory3} (ii) that $j_*k[I_{s,p}] \in \mathcal{L}$. 
Repeating this construction for every non-zero germ at $p$ yields a morphism
\[
  \bigoplus_{0 \not= s_p \in \mathcal{F}_p} j_*k[I_{s,p}] \to \mathcal{F}
\]
that is surjective on the stalk at $p$. 
Taking the direct sum over all points of $\Supp(\mathcal{F})$, we obtain an epimorphism 
\[
  \bigoplus_{p \in \Supp(\mathcal{F})} \bigoplus_{0 \not= s_p \in \mathcal{F}_p} j_*k[I_{s,p}] \to \mathcal{F}. 
\]
Since $\mathcal{L}$ is closed under arbitrary direct sums, the domain of this morphism is in $\mathcal{L}$. 
As $\mathcal{L}$ is also closed under quotients, it follows that $\mathcal{F} \in \mathcal{L}$.
Hence, $(\Supp^{-1} \circ \Supp)(\mathcal{L}) \subseteq \mathcal{L}$. Together with the reverse inclusion, this proves the claim.
\end{proof}

\begin{corollary}\label{Cor:LocalizingSubCatAndSkulaOpenR}
Let $P = \Rset \cup \{ \infty \}$. 
Then there are one-to-one correspondences 
\[
  \left\{ \begin{tabular}{c} Localizing \\ subcategories of $\Sh(P^{\sigma})$ \end{tabular}  \right\} 
  \xrightleftharpoons[\Supp^{-1}]{\Supp}  
  \left\{ \begin{tabular}{c} Skula-open \\ sets of $P^{\sigma}$ \end{tabular}  \right\} 
  \xrightleftharpoons[\sky^{-1}]{\sky} 
  \left\{ \begin{tabular}{c} Gabriel open \\ sets of $\Sp(P^{\sigma})$ \end{tabular}  \right\}.
\]
\end{corollary}

\begin{proof}
For the first bijection, combine Proposition \ref{Prop:OneSideInverseForSuppSets} and Theorem \ref{Thm:OneSideInverseForLocSubCat}. 
The second one is a consequence of Theorem \ref{Thm:ForSobSpaceGabrielIsSkula}. 
\end{proof}

%%%%%%%%%%%%%%%%%%%%%%%%%%%%%%%%%%%%%%%%%%%%%%%%%%%%%%%%%%
\subsection{Induced Ziegler Spectra}

In this subsection $P = \Rset$. 
Lerch \cite{LerchSpectrum} showed that the category of persistence modules $\Fun(\Rset, \Mod)$ is locally coherent, and that its Ziegler spectrum is homeomorphic to $\Idl(\Rset)$ endowed with the order topology (i.e.\ the topology generated by the rays $\{ D \mid D \subsetneq F \}$ and $\{ D \mid F \subsetneq D \}$, where $F \in \Idl(\Rset)$). For brevity, we write $\Zg(\Rset^a)$ for the Ziegler spectrum of $\Fun(\Rset, \Mod)$.

In contrast, it is easy to see that $\Sh(\Rset^{\sigma})$ is neither locally coherent nor even locally finitely presented.

\begin{proposition}\label{Prop:Easy}
Let $P = \Rset$. 
Then the following hold: 
\begin{enumerate}[(i)]
    \item the Scott sheaf $j_*k[p,q]$ is not finitely presented for any $p < q$; 
    \item $j^* \mathcal{F}$ is a finitely presented persistence module over $\Rset$, when $\mathcal{F}$ is a finitely presented Scott sheaf on $\Rset^{\sigma}$; 
    \item a Scott sheaf $\mathcal{F}$ on $\Rset^{\sigma}$ is finitely presented if and only if $\mathcal{F} = 0$; 
    \item $\Sh(\Rset^{\sigma})$ is not locally finitely presented. 
\end{enumerate}
\end{proposition}

Recall that in \cite[Proposition 3.4]{EphModAndScottShOverContPoset}, we observed that the direct image functor $j_* \colon \Sh(P^a) \to \Sh(P^{\sigma})$ along the identity map $j \colon P^a \to P^{\sigma}$ admits both a left and a right adjoint, denoted by $j^*$ and $j^!$, respectively.

\begin{proof}
(i) 
Consider the directed system $(j_*k[x,q])_{p < x < q}$ of subsheaves of $j_*k[p,q]$. 
Clearly, 
\[
  \Hom(j_*k[p,q], j_*k[x,q]) 
  = \Hom(k[p,q], j^!j_*k[x,q]) 
  = \Hom(k[p,q], k[x,q[) 
  = 0
\]
for all $p < x < q$. 
But 
\[
  \Hom \left(j_*k[p,q], \varinjlim_{p < x < q} j_*k[x,q] \right) 
  = \Hom(j_*k[p,q], j_*k[p,q]) 
  \cong k. 
\]
Hence, $j_*k[p,q]$ is not finitely presented.

(ii) 
Let $\mathcal{F}$ be a finitely presented Scott sheaf on $\Rset$, and let $(M_{i})_{i \in I}$ be a directed system of persistence modules over $\Rset$. 
Since $(j^*, j_*, j^!)$ is an adjoint triple and $j_*$ preserves direct limits, we have 
\begin{align*}
  \Hom \left(j^* \mathcal{F}, \varinjlim_{i \in I} M_i \right) 
  &= \Hom \left(\mathcal{F}, j_* \varinjlim_{i \in I} M_i \right) 
  = \Hom \left(\mathcal{F}, \varinjlim_{i \in I} j_*M_i \right) \\
  &= \varinjlim_{i \in I} \Hom(\mathcal{F}, j_*M_i) 
  = \varinjlim_{i \in I} \Hom(j^*\mathcal{F}, M_i). 
\end{align*}

(iii) 
Let $\mathcal{F}$ be a finitely presented Scott sheaf on $\Rset$. 
By (ii), $j^* \mathcal{F}$ is a finitely presented persistence module over $\Rset$. 
Lerch showed that $j^* \mathcal{F}$ must then be of the form $\bigoplus_{i = 1}^n k[p_i, q_i[$ for some $p_i < q_i$ ($i = 1, \dots, n$) (\cite[Lemma 3.6]{LerchSpectrum}). 
Thus, 
\[
  \mathcal{F} 
  = j_* j^* \mathcal{F} 
  = j_* \bigoplus_{i = 1}^n k \left[ p_i, q_i \right[ 
  = \bigoplus_{i = 1}^n j_*k[p_i, q_i]. 
\]
By (i), this cannot be finitely presented unless $\mathcal{F} = 0$. 

(iv) This is an immediate consequence of (iii).
\end{proof}

As a consequence of the previous proposition, there is no true Ziegler spectrum for Scott sheaves on $\Rset^{\sigma}$. 
Nevertheless, by \cite[Theorem~3.18]{EphModAndScottShOverContPoset}, the category of Scott sheaves is equivalent to the category of upper and lower semicontinuous persistence modules. Hence we may regard $\InjSpec(\Rset^\sigma)$ as a subset of $\InjSpec(\Rset^a)$ and endow it with the induced subspace topology. 

There are, in fact, two natural ways to do this, corresponding to upper and lower semicontinuous persistence modules. 
We define \[ \Zg_c(\Rset^\sigma) \coloneq \Zg(\Rset^a)\cap\Fun_c(\Rset,\Mod) \] and \[ \Zg^c(\Rset^\sigma) \coloneq \Zg(\Rset^a)\cap\Fun^c(\Rset,\Mod). \] We refer to these spaces as the \emph{open induced Ziegler spectrum} and the \emph{closed induced Ziegler spectrum}, respectively.

Identifying $\Zg(\Rset^a)$ with the homeomorphic space $\Idl(\Rset)$ as in \cite[Theorem~7.2]{LerchSpectrum}, we see that the underlying set of $\Zg_c(\Rset^\sigma)$ (resp.\ $\Zg^c(\Rset^\sigma)$) consists precisely of the ideals of $\Rset$ that are open (resp.\ closed) in the standard topology. Besides $\Rset$ itself, the closed ideals of $\Rset$ are the intervals $]-\infty,x]$ for $x\in\Rset$, whereas the open ideals are the intervals $]-\infty,x[$ for $x\in\Rset$.

Thus both $\Zg_c(\Rset^\sigma)$ and $\Zg^c(\Rset^\sigma)$ may be identified with the poset $\Rset\cup\{\infty\}$. These identifications are given by the natural bijections 
\[
  I_c\colon \Rset\cup\{\infty\}\to \Zg_c(\Rset^\sigma) 
  \qquad\text{and}\qquad 
  I^c\colon \Rset\cup\{\infty\}\to \Zg^c(\Rset^\sigma), 
\] 
defined by 
\[
  I_c(x)=k]-\infty,x[ \qquad\text{and}\qquad I^c(x)=k]-\infty,x] 
\] 
for $x\in\Rset$, and 
\[ 
  I_c(\infty)=k[\Rset]=I^c(\infty). 
\]

\begin{lemma}\label{Lemma:IndZieglerSubbasis}
\hfill
\begin{enumerate}[(i)]
    \item The open induced Ziegler spectrum $\Zg_c(\Rset^{\sigma})$ is generated by the families of sets 
    
    \[\{ k\left]-\infty, x\right[ \in \Zg_c(\Rset^{\sigma}) \mid x \le p \} \quad \text{and} \quad \{ k\left]-\infty, x\right[ \in \Zg_c(\Rset^{\sigma}) \mid x > p \},\] where $p \in \Rset$; 
    \item The closed induced Ziegler spectrum $\Zg^c(\Rset^{\sigma})$ is generated by the families of sets 
    \[\{ k\left] -\infty, x\right] \in \Zg^c(\Rset^{\sigma}) \mid x < p \} \quad \text{and} \quad \{ k\left] -\infty, x\right] \in \Zg^c(\Rset^{\sigma}) \mid x \ge p \},\] where $p \in \Rset$. 
\end{enumerate}
\end{lemma}

\begin{proof}
(i) 
The Ziegler topology of persistence modules is generated by the families of sets $\{ k[D] \in \InjSpec(\Rset^a) \mid D \subsetneq F \}$ and $\{ k[D] \in \InjSpec(\Rset^a) \mid F \subsetneq D \}$, where $F \subseteq \Rset$ is an ideal. 

Suppose first that $F$ is open in the standard topology, so that $F = \left]-\infty, p\right[$ for some $p \in \Rset \cup \{ \infty \}$. The intersections with $\Zg_c(\Rset^{\sigma})$ are then simply 
\begin{align*}
  \{ k[D] \in \InjSpec(\Rset^a) \mid D \subsetneq \left] -\infty, p\right[ \,\} \cap \Zg_c(\Rset^{\sigma}) 
  &= \{ \, k\left]-\infty, x\right[ \in \Zg_c(\Rset^{\sigma}) \mid \left]-\infty, x\right[ \subsetneq \left]-\infty, p\right[ \, \} \\
  &= \{ \, k\left]-\infty, x\right[ \in \Zg_c(\Rset^{\sigma}) \mid x < p \}
\end{align*}
and 
\begin{align*}
  \{ k[D] \in \InjSpec(\Rset^a) \mid \left] -\infty, p\right[ \subsetneq D \} \cap \Zg_c(\Rset^{\sigma}) 
  &= \{ k\left]-\infty, x\right[ \in \Zg_c(\Rset^{\sigma}) \mid \left]-\infty, p\right[ \subsetneq \left]-\infty, x\right[ \, \} \\
  &= \{ k\left]-\infty, x\right[ \in \Zg_c(\Rset^{\sigma}) \mid p < x \}. 
\end{align*}

Now suppose that $F=\left]-\infty,p\right]$ for some $p \in \Rset$. Then  
\begin{align*}
  \{ k[D] \in \InjSpec(\Rset^a) \mid \left] -\infty, p\right] \subsetneq D \} \cap \Zg_c(\Rset^{\sigma}) 
  &= \{ k\left]-\infty, x\right[ \in \Zg_c(\Rset^{\sigma}) \mid \left]-\infty, p\right] \subsetneq \left]-\infty, x\right[ \, \} \\
  &= \{ k\left]-\infty, x\right[\in \Zg_c(\Rset^{\sigma}) \mid p < x \} 
\end{align*}
and 
\begin{align*}
  \{ k[D] \in \InjSpec(\Rset^a) \mid D \subsetneq \left] -\infty, p\right] \, \} \cap \Zg_c(\Rset^{\sigma}) 
    &= \{ k\left]-\infty, x\right[ \in \Zg_c(\Rset^{\sigma}) \mid \left]-\infty, x\right[ \subseteq \left]-\infty, p\right[ \, \} \\ 
  &= \{ k\left]-\infty, x\right[ \in \Zg_c(\Rset^{\sigma}) \mid x \le p \}. 
\end{align*}

From these computations, we see that the families of sets \[ \{ k\left]-\infty, x\right[ \in \Zg_c(\Rset^{\sigma}) \mid x \le p \}\quad\text{and}\quad \{ k\left]-\infty, x\right[ \in \Zg_c(\Rset^{\sigma}) \mid x > p \}, \] where $p \in \Rset$, generate the open induced Ziegler topology on $\Zg_c(\Rset^{\sigma})$.

(ii) The argument is analogous.
\end{proof}

The generating families obtained in Lemma~\ref{Lemma:IndZieglerSubbasis} allow us to identify the open and closed induced Ziegler spectra with two natural topologies on $\Rset\cup\{\infty\}$.

\begin{theorem}\label{Thm:InducedZieglerSpectrumsHomeo} Let $P=\Rset$. 
\begin{enumerate}[(i)] 
    \item The map \[ I_c \colon \Rset\cup\{\infty\}\to\Zg_c(\Rset^\sigma) \] is a homeomorphism, where $\Rset\cup\{\infty\}$ is endowed with the Sorgenfrey topology, that is, the topology generated by the intervals $]x,y]$ with $x<y$. 
    
    \item The map \[ I^c \colon \Rset\cup\{\infty\}\to\Zg^c(\Rset^\sigma) \] is a homeomorphism, where $\Rset\cup\{\infty\}$ is endowed with the topology generated by the intervals $[x, \infty]$ ja $\left] - \infty , y\right[$. \end{enumerate} \end{theorem}

\begin{proof}
(i)
By Lemma~\ref{Lemma:IndZieglerSubbasis} (i), the generating families of $\Zg_c(\Rset^\sigma)$ correspond under $I_c$ to the families \[ \{\, \left]-\infty,p \right] \mid p\in\Rset\,\} \quad\text{and}\quad \{\, \left]p,\infty\right[\mid p\in\Rset\,\}, \] which generate the Sorgenfrey topology.
Hence $I_c$ is a homeomorphism.

(ii) The proof is analogous.
\end{proof}

\begin{corollary}\label{Cor:OpenIndZieglerHomeoGabriel}
Let $P = \Rset$. 
The open induced Ziegler spectrum and the Gabriel spectrum are homeomorphic: 
\[
  \Zg_c(\Rset^{\sigma}) \approx \Sp(\Rset^{\sigma}). 
\]
\end{corollary}

\begin{proof}
By Theorems \ref{Thm:GabrielSpecOfScottShOverR} and \ref{Thm:InducedZieglerSpectrumsHomeo} (i), both the Gabriel spectrum $\Sp(\Rset^{\sigma})$ and the open induced Ziegler spectrum $\Zg_c(\Rset^{\sigma})$ are homeomorphic to $\Rset\cup\{\infty\}$ endowed with the Sorgenfrey topology. 
\end{proof}

\begin{remark}
Let $P = \Rset$. 
Note that the Ziegler spectrum and the Gabriel spectrum of $\Fun(\Rset, \Mod)$ are not homeomorphic: 
\[
  \Zg(\Rset^a) \not\approx \Sp(\Rset^a). 
\]
The difference is caused by the ephemeral modules. 
Indeed, Proposition \ref{Prop:GabrielSpecOfDenseLinPoset} shows that the Gabriel spectrum contains many open points due to ephemeral modules. 
However, ephemeral modules are not finitely presented (\cite[Lemma 3.6]{LerchSpectrum}), and therefore they do not contribute to the Ziegler spectrum.
\end{remark}

%%%%%%%%%%%%%%%%%%%%%%%%%%%%%%%%%%%%%%%%%%%%%%%%%%%%%%%%%
%\bibliographystyle{plain}
%\bibliography{mybib}
\printbibliography

\end{document}